\documentclass[oneside]{amsart}
\usepackage[utf8]{inputenc}
\usepackage[bottom=0.5in]{geometry}
\usepackage{amsmath, amssymb, amsthm, graphicx, fancyhdr, textcomp, enumerate, diagbox, tcolorbox, esvect, tikz, adjustbox, todonotes, lipsum, listings, appendix,mathtools, algorithm, algpseudocode, algorithmicx, caption}

\algdef{SE}[DOWHILE]{Do}{doWhile}{\algorithmicdo}[1]{\algorithmicwhile\ #1}%
\usepackage{hyperref}
\hypersetup{
    colorlinks = true,
    allbordercolors = {white},
    allcolors = {blue}
}

\makeatletter
\def\BState{\State\hskip-\ALG@thistlm}
\makeatother
\algnewcommand\algorithmicinput{\textbf{Input:}}
\algnewcommand\algorithmicoutput{\textbf{Output:}}
\algnewcommand\Input{\item[\algorithmicinput]}%
\algnewcommand\Output{\item[\algorithmicoutput]}%

\usetikzlibrary{cd,arrows,shapes,positioning}
\usetikzlibrary{calc}

\newcommand{\Z}{\mathbb{Z}} 

\newcommand{\Q}{\mathbb{Q}}

\newcommand{\C}{\mathbb{C}}

\newcommand{\Aut}{\mathrm{Aut}} 

\newcommand{\stab}{\mathrm{stab}} 
\newcommand{\orb}{\mathrm{orb}} 

\newcommand{\legendre}[2]{\ensuremath{\left( \frac{#1}{#2} \right) }}

\newcommand{\GL}{\operatorname{GL}}
\newcommand{\SL}{\operatorname{SL}}

\newcommand{\lcm}{\operatorname{lcm}}
\newcommand{\Gal}{\operatorname{Gal}}

\newcommand{\End}{\operatorname{End}}

\newcommand{\jac}{\mathrm{Jac}}
\newcommand{\pic}{\mathrm{Pic}}

\newcommand{\genus}{\mathrm{genus}}
\newcommand{\Supp}{\mathrm{Supp}}

\newcommand{\im}{\operatorname{im}} 
\newcommand{\Mod}[1]{\ (\mathrm{mod}\ #1)} 
\renewcommand{\phi}{\varphi}

\NewDocumentCommand{\myrule}{O{1pt} O{3pt} O{black}}{%
  \par\nobreak 
  \kern\the\prevdepth 
  \kern#2 
  {\color{#3}\hrule height #1 width\hsize} 
  \kern#2 
  \nointerlineskip 
}

\usepackage{xcolor}

\newtheorem{theorem}{Theorem}
\newtheorem{lemma}[theorem]{Lemma}
\newtheorem{proposition}[theorem]{Proposition}
\newtheorem{remark}[theorem]{Remark}

\newtheorem{definition}[theorem]{Definition}
\newtheorem{corollary}[theorem]{Corollary}
\newtheorem{claim}{Claim}
\theoremstyle{definition}
\newtheorem{example}[theorem]{Example}

\makeatletter
\def\@seccntformat#1{\@ifundefined{#1@cntformat}%
   {\csname the#1\endcsname\quad}  
   {\csname #1@cntformat\endcsname}
}
\let\oldappendix\appendix 
\renewcommand\appendix{%
    \oldappendix
    \newcommand{\section@cntformat}{\appendixname~\thesection\quad}
}
\makeatother
\usepackage{appendix}
\algrenewcommand\algorithmicrequire{\textbf{Input:}}
\algrenewcommand\algorithmicensure{\textbf{Output:}}
\makeatletter
\newenvironment{breakablealgorithm}
  {
   \begin{center}
     \refstepcounter{algorithm}
     \hrule height.8pt depth0pt \kern2pt
     \renewcommand{\caption}[2][\relax]{
       {\raggedright\textbf{\fname@algorithm~\thealgorithm} ##2\par}%
       \ifx\relax##1\relax 
         \addcontentsline{loa}{algorithm}{\protect\numberline{\thealgorithm}##2}%
       \else 
         \addcontentsline{loa}{algorithm}{\protect\numberline{\thealgorithm}##1}%
       \fi
       \kern2pt\hrule\kern2pt
     }
  }{
     \kern2pt\hrule\relax
   \end{center}
  }
  \makeatother

\AtEndDocument{%
  \par
  \medskip
  \begin{tabular}{@{}l@{}}%
    \textsc{Wake Forest University, Winston-Salem, NC 27104}\\
    \textit{E-mail address}: \texttt{leemh23@wfu.edu}
    \\\\
    \textsc{University of California, Santa Barbara, Santa Barbara, CA 93106}\\
    \textit{E-mail address}: \texttt{meghanlee@ucsb.edu}
  \end{tabular}}

\usepackage{url}

\title{Isolated $j$-invariants arising from the modular curve $X_0(n)$ }
\author{Meghan Lee}
\date{}

\begin{document}

\newpage
\vspace{-.5cm}
\begin{abstract}
    An isolated point of degree $d$ is a closed point on an algebraic curve which does not belong to an infinite family of degree $d$ points that can be parameterized by some geometric object. We provide an algorithm to test whether a rational, non-CM $j$-invariant gives rise to an isolated point on the modular curve $X_0(n)$, for any $n \in \mathbb{Z}^+$, using key results from Menendez \cite{Zonia} and Zywina \cite{zywina}. This work is inspired by the prior algorithm of Bourdon et al. \cite{bourdon2024classification} which tests whether a rational, non-CM $j$-invariant gives rise to an isolated point on any modular curve $X_1(n)$. From the implementation of our algorithm, we determine that the set of $j$-invariants corresponding to isolated points on $X_1(n)$ is neither a subset nor a superset of those corresponding to isolated points on $X_0(n)$.
\end{abstract}
\maketitle

\section{Introduction}
The modular curve $X_1(n$) is an algebraic curve over $\Q$ whose non-cuspidal points parametrize elliptic curves  with a distinguished point of order $n$. In a seminal result in 1996, Merel \cite{merel} proved the uniform boundedness theorem which states that for a fixed degree $d$, there are no non-cuspidal degree $d$ points on $X_1(n)$ for $n$ sufficiently large. The classification of degree $d$ points on $X_1(n)$ is known for $d \leq 4$; see \cite{mazurdegleq3, kamiennydegleq3, kenkudegleq3, kamienny2degleq3, derickxdegleq3, classificationdeg4points}.

Less is known for the modular curve $X_0(n)$, which is an algebraic curve over $\Q$ whose non-cuspidal points parametrize elliptic curves with a distinguished cyclic subgroup of order $n$. The finite set of $n$ for which $X_0(n)$ has a non-cuspidal rational point is completely determined, by results from Mazur \cite{mazurX0n}, Kenku \cite{kenkuX0n}, and others; for more history, see \cite[Section 9]{lozano-robledo}. As for quadratic points of $X_0(n)$, a direct analogue of Merel's result is false, since non-cuspidal points of degree $2$ corresponding to elliptic curves with complex multiplication (CM) occur for infinitely many $n$. It is known for which $n$ there are infinitely many degree 2 points on $X_0(n)$, by work of Harris and Silverman \cite{harrissilverman}, Ogg \cite{oggpaper}, and Bars \cite{bars}. However, the set of $n$ for which there exists a nonempty and finite set of quadratic, non-CM, non-cuspidal points on $X_0(n)$ — which are isolated, degree 2 points — remains unknown. A recent conjecture of Mazur and Balakrishnan \cite[Conjecture 18]{ogg} states that there are finitely many such $n$. These questions motivate our work toward classifying $j$-invariants which arise from non-CM isolated points on $X_0(n)$.

Let $k$ be a number field. In general, a closed point of degree $d$ on an algebraic curve $C/k$ is \textit{isolated} if it does not belong to any infinite family of degree $d$ points that can be parameterized by some geometric object; i.e., parametrization by $\mathbb{P}^1$ or by a positive rank abelian subvariety of $\jac(C)$. See Section \ref{subsec:isolated} for details. If $x \in X_0(n)$ is an isolated point corresponding to the elliptic curve $E$ with a cyclic subgroup of order $n$, then $j(E) = j(x)$ and we call $j(x) \in X_0(1)$ an \textit{isolated $j$-invariant}. We develop an algorithm to test whether a given rational, non-CM $j$-invariant is isolated, and we show in Theorem \ref{CMthm} that every CM $j$-invariant is isolated. We rely on Zywina's \cite{zywina} algorithm for computing the adelic image of $E/\Q$, and a key result of Menendez \cite{Zonia}. The latter result appears in Menendez's PhD thesis \cite{Zonia}, but since the resulting paper is not yet available on arXiv and the conventions of our work are somewhat inconsistent, we include a general statement and proof for ease of reference in Appendix \ref{appendix:Menendez}, written by Abbey Bourdon.

\begin{algorithm}
\caption{NonIsolated}\label{introalgorithm:main}
    \begin{algorithmic}[1]
        \Require{A non-CM $j$-invariant $j \in \Q$.}
        \Ensure{A finite list \verb|{* <a_1, d_1>, ..., <a_k, d_k> *}| of $($level, degree$)$ pairs such that $j$ is isolated if and only if there exists an isolated point $x \in X_0(a_i)$ of degree $d_i$ with $j(x) = j$ for some pair \verb|<a_i, d_i>| in the list.}
        \end{algorithmic}
\end{algorithm}
\setcounter{algorithm}{0} 

If the output of Algorithm \ref{introalgorithm:main} is empty for a given $j \in \Q$, then $j$ is not an isolated $j$-invariant with respect to $X_0(n)$. If the output is nonempty, then other techniques must be used to determine whether $j$ is isolated.

If $E$ is an elliptic curve with a mod $n$ image $G(n)$ of genus 0, then the modular curve $X_{G(n)}$ is isomorphic to $\mathbb{P}^1$. As part of Algorithm \ref{introalgorithm:main}, we show that in this case, we can construct a degree $d$ map from $X_0(n)$ to $\mathbb{P}^1$ which parametrizes any point $x \in X_0(n)$ with the same $j$-invariant as $E$; see Section \ref{section:genus0}.
\begin{theorem}\label{genus0inintro}
    Let $E/\Q$ be an elliptic curve with mod $n$ image $G(n)$ of genus $0$, where $n$ is a positive integer. Then every $x \in X_0(n)$ with $j$-invariant $j(x) = j(E)$ is $\mathbb{P}^1$-parameterized.
\end{theorem}

Thus by Theorem \ref{genus0inintro}, we can restrict to running the code associated to Algorithm \ref{algorithm:main} on all non-CM $j$-invariants in the LMFDB associated to an adelic image of genus strictly greater than 0. The resulting output shows that the 14 $j$-invariants of this list which correspond to a known isolated non-CM rational point on some $X_0(n)$ are the only isolated $j$-invariants from this list; see Lozano-Robledo \cite[Table 4]{lozano-robledo}. We list these $j$-invariants in Table \ref{table1}.

\begin{center}
\captionof{table}{Known rational non-CM isolated $j$-invariants for $X_0(n)$}\label{table1}
\begin{tabular}{||c c c||} 
 \hline
 $j$-invariant & level & degree of isolated point \\ [0.5ex] 
 \hline
  \hline
 -121 & 11 & 1 \\
 \hline
 -24729001 & 11 & 1\\
 \hline
 46969655/32768 & 15 & 1 \\
 \hline
-121945/32 & 15 & 1 \\
\hline
-349938025/8 & 15 & 1 \\
\hline
-25/2 & 15 & 1\\
\hline
-882216989/131072 & 17 & 1 \\
 \hline
 -297756989/2 & 17 & 1 \\
 \hline
 3375/2 & 21 & 1 \\ 
 \hline
 -189613868625/128 & 21 & 1 \\
 \hline
-1159088625/2097152 & 21 & 1 \\
\hline
 -140625/8 & 21 & 1
 \\
 \hline
 -9317 & 37 & 1
 \\
 \hline -162677523113838677 & 37 & 1\\
 [1ex]
 \hline
\end{tabular}
\end{center}

It is natural to ask: Are the isolated $j$-invariants of $X_0(n)$ correlated to the isolated $j$-invariants of $X_1(n)$? Prior to this work, it was shown by Bourdon, Hashimoto, Keller, Klagsbrun, Lowry-Duda, Morrison, Najman, and Shukla 
\cite[Theorem 2]{bourdon2024classification} that only 3 of the 14 $j$-invariants in Table 1 correspond to isolated points on $X_1(n)$: $-140625/8, -9317,$ and $ -162677523113838677$. This was determined by the authors' algorithm which tests whether a non-CM $j \in \Q$ arises as the $j$-invariant of an isolated point on any modular curve $X_1(n)$. Conversely, in this paper, we show that there is a $j$-invariant corresponding to an isolated point on $X_1(n)$ which does not correspond to an isolated point on $X_0(n)$. The elliptic curve with $j$-invariant $351/4$ gives an isolated point of degree 9 on $X_1(28)$, as shown in \cite[Theorem 2]{odddegree}. We find that $j = 351/4$ does not correspond to an isolated point on $X_0(n)$ for any $n \in \Z^+$, which implies the following result.
\begin{theorem}\label{subsets result in intro}
    The set of rational $j$-invariants corresponding to isolated points on $X_1(n)$ is neither a subset nor a superset of those corresponding to isolated points on $X_0(n)$.
\end{theorem}

After developing our main algorithm, we learned that our computational results coincide with recent work of Terao \cite[Theorem 1.6]{terao}. Terao showed that, assuming a conjecture of Zywina \cite{zywina} concerning images of Galois representations of elliptic curves over $\Q$, the above $j$-invariants are the complete set of rational isolated $j$-invariants for $X_0(n)$. Our work gives a distinct, algorithmic approach for detecting isolated points in which we primarily work with the $m$-adic Galois representation, where $m$ denotes the product of 2, 3, and all non-surjective primes. By contrast, Terao makes use of Zywina's theory of agreeable closures \cite{zywina}. Additionally, Terao's work handles isolated points on geometrically disconnected curves in the proof of \cite[Theorem 1.3]{terao}, whereas our proof of Theorem \ref{thm38} utilizes a descent lemma of Derickx \cite{bourdon2024classification} to change the field of definition, ensuring that the relevant modular curves remain geometrically connected.

Our approach involves additional results concerning CM $j$-invariants and also relies on developing the theory of ``primitive points," characterized in Theorem \ref{introsprimitivecharacterization}, which was modeled after previous work of Bourdon, Hashimoto, Keller, Klagsbrun, Lowry-Duda, Morrison, Najman, and Shukla \cite{bourdon2024classification}.

\subsection{Key components of main algorithm}
The first step of Algorithm \ref{algorithm:main} applies Zywina's algorithm \cite{zywina} for computing the image of the adelic Galois representation of a non-CM elliptic curve $E/\Q$ as a subgroup of $\GL_2(\widehat{\Z}).$ Using the adelic image, we apply results from Menendez \cite[Theorem 5.1]{Zonia} and Bourdon, Ejder, Liu, Odumodu, and Viray \cite[Theorem 4.3]{belov} to compute a finite set of ``primitive points" associated to a non-CM elliptic curve. The following theorem characterizes primitive points.
\begin{theorem}\label{introsprimitivecharacterization}
    Let $E/\Q$ be a non-CM elliptic curve, and let $\mathcal{P}(E)$ denote the primitive points. Then:
    \begin{itemize}
    \item The set $\mathcal{P}(E)$ is finite.
    \item For every $n \in \Z^+$ and every closed point $x \in X_0(n)$ associated to $E$, there is a unique point in the set $\mathcal{P}(E)$ which corresponds to $x$ via the natural projection map.
    \item The rational number $j(E)$ is isolated if and only if there exists an isolated point in $\mathcal{P}(E)$.
    \end{itemize}
\end{theorem}
Thus, we can reduce the problem of searching for isolated points in the infinite set of curves, $\cup_{n \in \Z^+} X_0(n)$. In particular, since $j(E)$ is isolated if and only if there exists an isolated point in $\mathcal{P}(E)$, it is sufficient to test whether any points in $\mathcal{P}(E)$ are isolated. Primitive points on $X_1(n)$ were defined in \cite[Section 5]{bourdon2024classification}, but in this work we provide a simpler definition for primitive points on $X_0(n)$ and in Section \ref{subsec:primitiveptsconnections} we prove its equivalence to the definition given in \cite{bourdon2024classification}.

We next use the level and degree information of the primitive points to filter the set $\mathcal{P}(E)$. In doing so, we can eliminate points that we know to be non-isolated from the set of primitive points of $E$. First, if the degree of some closed point $x \in X_0(n)$ is strictly greater than the genus of the modular curve $X_0(n)$, then the Riemann-Roch space associated to $x$ is non-trivial, implying that $x$ is $\mathbb{P}^1$-parameterized; see Lemma \ref{thm:riemannrochlemma}. Second, by Theorem \ref{genus0inintro}, any point $x \in X_0(n)$ with the same $j$-invariant as an elliptic curve with a genus 0 mod $n$ image is $\mathbb{P}^1$-parametrized; see Theorem \ref{thm38}.

A closed point $x \in C$ of degree $d$ is \textit{sporadic} if there are finitely many closed points of $C$ with degree at most $d$. In the case of CM $j$-invariants, we prove in Section \ref{section:CM} that all CM $j$-invariants are sporadic, and thus are isolated, in the following theorem, which is a quick consequence of Theorem 7.3 in \cite{lemmalowdegreeissporadic}. Thus there is no need for an algorithm to test CM $j$-invariants.

\begin{theorem}
    Let $E$ be an elliptic curve with CM by an order in an imaginary quadratic field $K$. Then for all sufficiently large primes $\ell$ which split in $K$, there exists a sporadic point $x = (E, \langle P \rangle) \in X_0(\ell)$ with only sporadic lifts. Specifically, for any positive integer $m\in \Z^+$ and any point $y \in X_0(m\ell)$ with $f(y) = x$, the point $y$ is sporadic, where $f$ denotes the natural map $X_0(m \ell) \to X_0(\ell)$.
\end{theorem}

\subsection{Outline}
We first give relevant background in Section \ref{section:background}. Next, in Section \ref{section:CM}, we prove that all CM $j$-invariants are isolated. In Section \ref{section:genus0}, we prove Theorem \ref{genus0inintro}, which allows us to conclude that if there is an elliptic curve $E/\Q$ with a mod $n$ image of genus 0, then every $x \in X_0(n)$ with $j(x) = j(E)$ is $\mathbb{P}^1$-parameterized. In Section \ref{section:overview}, we give an overview of our main algorithm for testing whether a $j$-invariant is isolated. In Sections \ref{section:primitive} and \ref{section:primitivealgorithm}, we characterize the set of primitive points, connect our definition to the definition given in \cite{bourdon2024classification}, and describe a sub-algorithm for computing the set. We give a proof showing the validity of the main algorithm in Section \ref{section:validity}, and in Section \ref{section:351/4} we describe an example which proves Theorem \ref{subsets result in intro}.

\subsection{Code}
We have implemented our algorithm in Magma. Our code, and output from running our algorithm on all 30,141 non-CM $j$-invariants associated to an adelic image with genus 0, is available in the Github repository at \href{https://github.com/meghanhlee/NonIsolated}{https://github.com/meghanhlee/NonIsolated}.

\section*{Acknowledgments}
The author was partially supported by NSF grant DMS-2137659. We thank Abbey Bourdon, Jeremy Rouse, and Leandro Lichtenfelz for helpful conversations and comments on earlier versions of this paper. We also thank Jacob Mayle for helpful conversations.

\section{Background}\label{section:background}
\subsection{Conventions}
Let $k$ denote a number field, and $\overline{k}$ denote its algebraic closure. Throughout, we take $C$ to be a nice curve defined over $k$; i.e., $C$ is a smooth, projective, geometrically integral, 1-dimensional scheme defined over $k$.

We write the absolute Galois group of $k$ as $\Gal(\overline{k}/k)$. For a prime number $p$, we denote the ring of $p$-adic integers as $\Z_p$. If $m\in\Z^+$ is not prime, then $\Z_m \coloneqq \prod_{p \mid m} \Z_p$. For a positive integer $a \in \Z^+$, the set $\Supp(a)$ denotes the set of prime divisors of $a$. 

Let $\widehat{\Z}$ denote the ring of profinite integers. We often write $G(n)$ to be the image of a subgroup $G \leq \GL_2(\widehat{\Z})$ via the projection map $\GL_2(\widehat{\Z}) \to \GL_2(\Z/n\Z).$

We refer to the $L$-functions and Modular Forms Database as the LMFDB \cite{lmfdb}.

\subsection{Elliptic curves}
Let $k$ be a number field. Recall that an elliptic curve $E$ over $k$ is a nonsingular, projective, cubic curve with a distinguished point $\mathcal{O}$ given by the short Weierstrass equation
\begin{align*}
    E : Y^2 Z &= X^3 + aXZ^2 + bZ^3.
\end{align*}

We write $E[n]$ to denote the set of all torsion points with order dividing $n$. Throughout, we identify elements of $E[n]$ with column vectors.

Recall for any fixed number field $k$, there are finitely many isomorphism classes of elliptic curves with an endomorphism ring larger than $\Z$, and we say such elliptic curves have complex multiplication.
\begin{definition}
    An elliptic curve $E$ defined over a number field $k$ has \textbf{complex multiplication} (CM) if its endomorphism ring is larger than $\Z.$ In particular, there is an imaginary quadratic field $K$ such that
    \begin{align*}
        \End_{\overline{k}}(E) \cong \mathcal{O},
    \end{align*}
    where $\mathcal{O}$ is the order in $K$ of conductor $f.$
\end{definition}

\begin{definition}
    Let $\mathcal{O}$ denote the ring of integers in the field $K$. The \textbf{discriminant of an order} in an imaginary quadratic field is defined as $\Delta = f^2 \Delta_K$, where the conductor $f = [\mathcal{O}_K : \mathcal{O}]$ and $\Delta_K$ denotes the discriminant of the field $K.$
\end{definition}

\subsection{The modular curve $X_0(n)$}
Here we recall the construction of the modular curve $X_0(n).$ For any $n \in \Z^+,$ we define the congruence subgroup of $\SL_2(\Z)$,
\begin{align*}
    \Gamma_0(n) \coloneqq \left\{ \begin{pmatrix}
        a & b \\ c & d
    \end{pmatrix} \in \SL_2(\Z) : c \equiv 0 \Mod{n}\right\}.
\end{align*}
This congruence subgroup acts on $\mathbb{H}$, the upper half-plane of $\C$, via \textbf{linear fractional transformations}; i.e., if $\gamma \in \Gamma_0(n)$, and $\tau \in \mathbb{H}$, then the action $\gamma : \mathbb{H} \to \mathbb{H}$ is given by
    \begin{align*}
        \gamma(\tau) = \frac{a\tau + b}{c\tau + d}.
    \end{align*}
The quotient $\mathbb{H}/\Gamma_0(n)$ is a Riemann surface, whose points correspond to isomorphism classes of elliptic curves over $\C$ with a distinguished point of order $n$, but it is not compact. So we then take its compactification, $\mathbb{H}^*/\Gamma_0(n)$, where $\mathbb{H}^*$ denotes the extended upper half-plane, which is given by
\begin{align*}
\mathbb{H}^* = \mathbb{H} \cup \mathbb{P}^1(\Q).
\end{align*}
We can then identify this Riemann surface with a smooth, projective curve defined over $\Q$, which is the modular curve $X_0(n)$. For details, see \cite[Theorem 13.1]{silverman}.

A point $x \in X_0(n)(k)$ is \textbf{noncuspidal} if it corresponds to a pair $(E, \langle P \rangle)$, where $E$ is an elliptic curve over $k$ and $\langle P \rangle$ is a distinguished $k$-rational cyclic subgroup of order $n$. That is, for every $\sigma \in \Gal(\overline{k}/k)$ and every $a \in \langle P \rangle$, we have $\sigma(a) \in \langle P \rangle$.

\subsection{Degree of projection map between modular curves}\label{section:degreeofmap}
We provide formulas for the degree of the natural projection map $f:X_0(n) \to X_0(m)$ for $n, m \in \Z^+$, such that $m$ is a proper divisor of $n$, since $f: X_0(n) \to X_0(n)$ is a map of degree 1. We use this degree computation in Algorithm \ref{algorithm:primitive} to compute the list of primitive points of $E$ characterized in Definition \ref{defn:myprimitivepoints}.

\begin{proposition}
    Let $a, b \in \Z^+$. There is a natural $\Q$-rational map $f: X_0(ab) \to X_0(a)$ sending $[E, \langle P \rangle]$ to $[E, \langle bP \rangle]$ with
    \begin{align*}
        \deg(f) = b \prod_{p \mid b, p \nmid a} \left( 1 + \frac{1}{p} \right).
    \end{align*}
    \begin{proof}
        This follows from the formulas of \cite[p. 66]{gamma0degrees}.
    \end{proof}
\end{proposition}

\subsection{Isolated points}\label{subsec:isolated}

\subsubsection{Preliminary definitions}

Let $C$ be a curve defined over a number field $k$. Throughout this section, we denote the Picard group by $\pic(C)$, the $dth$ symmetric power of a curve by $C^{(d)}$, and the Jacobian of a curve by $\jac(C)$.

The \textbf{closed points} of $C$ can be identified with the $\Gal(\overline{k}/k)$-orbits in $C(\overline{k})$. From this perspective, if $x \in C$ is a closed point, then the \textbf{degree} of $x$ is the size of the Galois orbit in $C(\overline{k}).$ Alternatively, the degree of $x$ is also defined as the degree of the field extension $[k(x) : k]$, where $k(x)$ is the residue field. More precisely, a closed point is a scheme-theoretic point whose Zariski closure is equal to itself.

Let $x \in C$ be a closed point of degree $d$. We have that each degree $d$ closed point $x \in C$ corresponds to a $k$-rational effective divisor
\begin{align*}
    P_1 + \dots + P_d,
\end{align*}
such that $P_1, \dots, P_d$ are the points in the $\Gal(\overline{k}/k)$-orbit of $x$. In the definitions that follow, we consider the image of a closed point $x$ under the natural map from $C^{(d)}$ to $\jac(C)$. Denote this map
\begin{align*}
    \Phi_d : C^{(d)} \to \jac(C).
\end{align*}
We assume there exists $P_0 \in C(k)$, though this is not required; see \cite{belov}. The map $\Phi_d$ sends the effective divisor $D$ of degree $d$ to the equivalence class $[D - dP_0] \in \jac(C)$. The cases in which the map $\Phi_d$ is injective or non-injective on all closed points with degree $d$ each give rise to a notion of geometric parametrization. First, if $x \in C$ is of degree $d$ and $\Phi_d$ is non-injective on $C^{(d)}(k)$, then there exists a point $y \neq x \in C^{(d)}(k)$ such that $\Phi_d(x) = \Phi_d(y)$. This gives a non-constant function $f$ of degree $d$ in the function field of $C$ such that div$(f) = x- y$. For such a function, we know by Hilbert's Irreducibility Theorem \cite[Chapter 9]{serrelectures} that there are infinitely many closed points of degree $d$ in the pullback $f^{-1}(\mathbb{P}^1(k))$, thus inspiring the definition of $\mathbb{P}^1$-parametrization.
 
\begin{definition}
The degree $d$ closed point $x \in C$ is \textbf{$\mathbb{P}^1$-parameterized} if there exists a point $y \in C^{(d)}(k)$, with $y \neq x$, such that $\Phi_d(x) = \Phi_d(y)$. Otherwise, if such a point $y$ does not exist, then the point $x$ is called \textbf{$\mathbb{P}^1$-isolated}.
\end{definition}

If $\Phi_d$ is injective on all degree $d$ points, then by Faltings's Theorem \cite{faltings}, we know that all but finitely many degree $d$ points are associated to a positive rank abelian subvariety of $\jac(C)$, thus motivating the definition of AV-parametrization.

\begin{definition}
    The degree $d$ closed point $x \in C$ is \textbf{AV-parameterized} if there exists a positive rank abelian subvariety $A$ over $k$ with $A \subset \jac(C)$, such that $\Phi_d(x) + A\subset \im(\Phi_d).$ Otherwise, if such a subvariety does not exist, then the point $x$ is \textbf{AV-isolated}.
\end{definition}

If $\Phi_d$ does not admit a parametrization of $x \in C$ in either case, then we call the point $x$ isolated. 
\begin{definition}
    The degree $d$ closed point $x \in C$ is \textbf{isolated} if it is both $\mathbb{P}^1$-isolated and AV-isolated.
\end{definition}

This definition was motivated by an observation of Frey \cite{frey}, which noted that by Faltings's theorem, a curve over a number field may have infinitely many points of degree at most $d$ if and only if these points can be parameterized by a geometric object. In \cite{belov}, Bourdon, Ejder, Liu, Odumodu, and Viray showed that any such geometric parametrization will give rise to an infinite family of points of degree exactly $d$.

\begin{example}
    The affine equation for the modular curve $X_0(48)$ is given by
\begin{align*}
    X_0(48) : y^2 = x^8 + 14x^4 + 1.
\end{align*}
Consider the projection map $f: X_0(48) \to \mathbb{P}^1,$ which maps $(x, y) \mapsto x.$
This degree 2 map gives an infinite family of quadratic points which are $\mathbb{P}^1$-parameterized since, infinitely often, the pullback of a rational $x \in \mathbb{P}^1(\Q)$ will be a quadratic point on $X_0(48)$, by Hilbert's Irreducibility Theorem. An example of a quadratic point which is not parametrized by the map $f$ is $(\sqrt{-1}, 4)$; since $\sqrt{-1}\not\in \Q$, the point $(\sqrt{-1}, 4)$ is not in the set $f^{-1}(\mathbb{P}^1(\Q)).$ Bruin and Najman \cite{bruinnajman} proved that the complete set of isolated quadratic points is given by the set $\{(\pm \sqrt{-1}, 4), (\pm \sqrt{-1}, -4)\}$.
\end{example}

    \begin{definition}
    If $x = (E, \langle P \rangle) \in X_0(n)$ is isolated, then we call the $j$-invariant $j(E) \in X_0(1) \cong \mathbb{P}^1$ an \textbf{isolated $j$-invariant}.
\end{definition}

\subsubsection{Results characterizing isolated points}

Bourdon, Ejder, Liu, Odumodu, and Viray give the following characterization of isolated points.

\begin{theorem}\label{theorem 4.2 belov}\cite[Theorem 4.2]{belov}
    Let $C$ be an algebraic curve over a number field.
    \begin{itemize}
        \item[(1)] There are infinitely many degree $d$ points on $C$ if and only if there is a degree $d$ point on $C$ that is not isolated.
        \item[(2)] There are only finitely many isolated points on $C$.
    \end{itemize}
\end{theorem}

\begin{definition}
    A degree $d$ closed point $x \in $C is \textbf{sporadic} if there are only finitely many closed points of $C$ with degree at most $d.$ 
\end{definition}
We note that Theorem \ref{theorem 4.2 belov} implies that if there are finitely many points of degree $d$ on $C$, then each of the degree $d$ points is isolated. Thus it follows from Theorem $\ref{theorem 4.2 belov}$ that every sporadic point is isolated.

The following degree condition is a key result we use in Section \ref{section:primitive} to construct the set of (possibly isolated) primitive points associated to an elliptic curve $E$.
\begin{theorem}\cite[Theorem 4.3]{belov}\label{belovdeg}
    Let $C, D$ be algebraic curves, and let $f: C \to D$ be a finite map. Let $x \in C$ be a closed point. Assume the degree condition \begin{align*}
        \deg(x) = \deg(f) \cdot \deg(f(x))
    \end{align*}
    holds. If $x \in C$ is isolated, then $f(x) \in D$ is isolated. 
\end{theorem}

\subsubsection{Results characterizing parameterized points}

\begin{lemma}\label{lemma36} \cite[Lemma 36]{bourdon2024classification}
        Let $X/k$ be a curve. Let $x \in X^{(d)}(k)$ be an irreducible divisor of degree $d$. Then the following are equivalent:
    \begin{itemize}
        \item[(i)] The point $x$ is $\mathbb{P}^1$-parameterized.
        \item[(ii)] There is a non-constant morphism $\mathbb{P}^1 \to X^{(d)}$ which contains $x$ in its image.
    \end{itemize}
\end{lemma}

\begin{lemma}\label{lemma37} \cite[Lemma 37]{bourdon2024classification}
Let $f: X\to Y$ be a finite morphism of curves and $x$ a closed point on $X$, and assume that $\deg x = \deg (f(x))$. If the point $x$ is $\mathbb{P}^1$-parameterized, then $f(x)$ is also $\mathbb{P}^1$-parameterized.
\end{lemma}

We now provide a sufficient condition for a closed point $x \in C(\overline{k})$ to be $\mathbb{P}^1$-parameterized. The following lemma is used in Algorithm \ref{algorithm:main} to filter the initial set of possibly isolated points obtained from the sub-algorithm, Algorithm \ref{algorithm:primitive}.

\begin{lemma}\label{thm:riemannrochlemma} 
    Let $C$ be a curve of genus $g$ over a number field $k$. If $x \in C$ is a closed point with degree satisfying
    \begin{align*}
        \deg(x) \geq  1 + g,
    \end{align*}
    then $x$ is $\mathbb{P}^1$-parameterized.
    \begin{proof}
    Let $D$ be the effective divisor of degree $d$ that corresponds to $x \in C$ for which the inequality
    \begin{align*}\label{ineq:filterRR}
        \deg (x) &\geq g+1 
    \end{align*}
    holds. It then follows from the Riemann-Roch Theorem that
    \begin{align*}
        \ell (D) &\geq g + 1 + 1 - g \\
        &= 2,
    \end{align*}
    where $\ell(D)$ denotes the dimension of the Riemann-Roch space associated to $D$. Thus we know that the Riemann-Roch space associated to $x$ is nontrivial. In other words, there are at least 2 global sections which vanish at $x$, and in particular, these sections are linearly independent. Hence, there must exist a non-constant function $f : C \to \mathbb{P}^1$ with pole divisor $D$.

    Thus we can define $y \coloneqq f^*([0 : 1])$, so that
    \begin{align*}
        y - x
        &= \text{div}(f).
    \end{align*}
    Hence $\Phi(y) = \Phi(x)$, and so $x$ is $\mathbb{P}^1$-parameterized.
    \end{proof}
    \end{lemma}

\subsection{Galois representations of elliptic curves}
Let $E$ be an elliptic curve over a number field $k$. The action of the absolute Galois group $\Gal(\overline{k}/k)$ on the torsion points of $E$ is given by the adelic Galois representation.
\begin{definition}
    The homomorphism
    \begin{align*}
        \rho_E : \Gal(\overline{k}/k) \to \Aut(E(\overline{k})_{\text{tors}}) \cong \GL_2(\widehat{\Z})
    \end{align*}
    is the \textbf{adelic Galois representation} associated to $E/k$.
\end{definition}
The adelic Galois representation also gives rise to two other Galois representations. Firstly, if we fix $m \in \Z^+$, then the Galois group's action on points with order divisible only by primes dividing $m$ is described by the $m$-adic Galois representation associated to $E$.
\begin{definition}
    Fix $m \in \Z^+$. Let
    \begin{align*}
        f: \prod_{p \text{ prime}} \GL_2(\Z_p) \to \prod_{p \mid m} \GL_2(\Z_p)
    \end{align*} denote the natural projection map. The homomorphism
    \begin{align*}
        \rho_{E, m^{\infty}} : \Gal(\overline{k}/k) \xrightarrow{\rho_E} \GL_2(\widehat{\Z}) \cong \prod_{p \text{ prime}} \GL_2(\Z_p) \xrightarrow{f} \prod_{p \mid m} \GL_2(\Z_p),
    \end{align*}
    given by the composition $\rho_{E, m^{\infty}} = f \circ \rho_E$ is the \textbf{$m$-adic Galois representation} associated to $E.$ In particular, if $m = \ell$ is prime, then we have the standard $\ell$-adic Galois representation associated to $E$.
\end{definition}

Fix a positive integer $m \in \Z^+$ and a basis for $E[m]$, where $E[m]$ denotes the set of points in $E(\overline{k})$ with order dividing $m$. The action of $\Gal(\overline{k}/k)$ on $E[m]$ is described by the mod $m$ Galois representation.
\begin{definition}
    Fix $m \in \Z^+.$ The homomorphism
    \begin{align*}
        \rho_{E, m} : \Gal(\overline{k}/k) \to \Aut(E[m]) \cong \GL_2(\Z/m\Z)
    \end{align*} is the \textbf{mod $m$ Galois representation} associated to $E.$
\end{definition}

Serre's Open Image Theorem \cite{serre} states that if $E/k$ is a non-CM elliptic curve, then $\im \rho_E$ is open in $\GL_2(\widehat{\Z})$; as a consequence, there exists a positive integer $n \in \Z^+$ for which the adelic Galois representation of $E$ is the complete preimage of the mod $n$ image.
\begin{definition}
    The \textbf{level of the adelic Galois representation} of an elliptic curve $E$ is the smallest positive integer $n \in \Z^+$ such that $\im \rho_E = \pi^{-1} (\im \rho_{E, n})$, where $\pi : \GL_2(\widehat{\Z}) \to \GL_2(\Z/n\Z)$ denotes the natural reduction map.
\end{definition}

Fix $m \in \Z^+$. We can similarly define the level of the $m$-adic Galois representation of $E$.
\begin{definition}
    The \textbf{level of the $m$-adic Galois representation} of an elliptic curve $E$ is the smallest positive integer $n \in \Z^+$ such that $\im \rho_{E, m^\infty} = \pi^{-1}(\im \rho_{E,n}).$
\end{definition}

We can interpret a more general modular curve using Galois representations. If $H \leq \GL_2(\widehat{\Z})$ is an open subgroup, then the modular curve $X_H$ parametrizes elliptic curves $E$ whose adelic image $\im \rho_E(\Gal(\overline{k}/k))$ is contained in $H.$

\section{CM $j$-invariants are isolated}\label{section:CM}
In this section, we prove that there are infinitely many sporadic points associated to any CM $j$-invariant, thus giving us that all CM $j$-invariants are isolated.

\begin{theorem}\label{CMthm}
    Let $E$ be an elliptic curve with CM by an order in an imaginary quadratic field $K$. Then for all sufficiently large primes $\ell$ which split in $K$, there exists a sporadic point $x = (E, \langle P \rangle) \in X_0(\ell)$ with only sporadic lifts. Specifically, for any positive integer $m\in \Z^+$ and any point $y \in X_0(m\ell)$ with $f(y) = x$, the point $y$ is sporadic, where $f$ denotes the natural map $X_0(m \ell) \to X_0(\ell)$.
    \end{theorem}

\subsection{Preliminary results}

The following lemma was shown in \cite[Theorem 7.3]{lemmalowdegreeissporadic} for any subgroup $H_0$ of
\begin{align*}
    \GL_2(\Z/n\Z)/\{\pm 1\},
\end{align*} and effectively, it gives us that if a point is of sufficiently low degree, then it is sporadic. In this statement of the result, we take $H_0$ to be the subgroup of upper triangular matrices.
\begin{lemma}\cite[Theorem 7.3]{lemmalowdegreeissporadic}\label{lemma:lowdegree}
    Suppose there is a point $x \in X_0(n)$ such that
    \begin{align*}
        \deg(x) \leq \frac{119}{24000} \left(\deg\left(X_0(n) \to X_0\left(1\right)\right)\right).
    \end{align*}
    Then $x$ is sporadic. Moreover, if we have any $d \in \Z^+$ and $f: X_0(dn) \to X_0(n)$ denoting the natural projection map, and if there is $y \in X_0(dn)$ with $f(y) = x$, then $y$ is sporadic.
\end{lemma}

\subsection{Proof of Theorem \ref{CMthm}}

Fix a Weierstrass equation for $E/K(j(E)).$ The degree of the extension $[\Q(j(E)) : \Q]$ is given by $h(\mathcal{O})$, the class number of $\mathcal{O}$. We have the following diagram of field extensions:
        \[\begin{tikzcd}
	& {K(j(E))} \\
	{\mathbb{Q}(j(E))} & K \\
	{\mathbb{Q}}
	\arrow["{h(\mathcal{O})}", no head, from=1-2, to=2-2]
	\arrow["2", no head, from=2-1, to=1-2]
	\arrow["{h(\mathcal{O})}"', no head, from=2-1, to=3-1]
	\arrow["2"', no head, from=3-1, to=2-2]
        \end{tikzcd}\]
        where the degree of the extension $[K(j(E)):\Q]$ is fixed.

        Recall that in a fixed imaginary quadratic field, there exist infinitely primes that split. So choose a split prime $\ell$ such that $\legendre{\Delta}{\ell} = 1$ and
        \begin{align*}
            \ell > \frac{48000}{119}h(\mathcal{O})-1.
        \end{align*}
        Equivalently, $2h(\mathcal{O}) < \frac{119}{24000}(\ell+1)$.
        
        By \cite{CMbourdonclark}, we know that for any prime $\ell$ that splits in $\mathcal{O}$, the elliptic curve $E$ defined over $K(j(E))$ has a cyclic $\ell^k$-isogeny for any $k \in \Z^+$. Thus, there is a point $P \in E(\overline{\Q})$ of order $\ell^k$ for any $k \in \Z^+$, such that $\langle P \rangle$ is defined over $K(j(E)).$ In particular, $[K(j(E), \langle P \rangle) : K(j(E))] = 1$. Then for the point $x = (E, \langle P \rangle) \in X_0(\ell^k)$, we have
        \begin{align*}
            \deg(x) &\leq [K(j(E), \langle P \rangle) : \Q] \\
            &= [K(j(E), \langle P \rangle):K(j(E))][K(j(E): K] [K:\Q] \\
            &= 2 h(\mathcal{O}) \\
            &< \frac{119}{24000}(\ell+1).
        \end{align*}
Now note that $\ell+1 \leq \deg(X_0(\ell^k) \to X_0(1)),$ and so we have satisfied the inequality given in Lemma \ref{lemma:lowdegree}. As such, the point $x$, and every lift of $x$, is sporadic. Thus $x$, and every lift of $x$, is isolated.

\section{Genus 0 adelic images do not correspond to isolated points}\label{section:genus0}

In this section, we prove in Theorem \ref{thm38} that all elliptic curves with a mod $n$ image $G(n)$ of genus 0 — that is, where $X_{G(n)}$ is a modular curve of genus 0 — correspond to $\mathbb{P}^1$-parameterized points on the modular curve $X_0(n)$, and thus are not isolated.

\subsection{Preliminary results}
We begin with the following lemmas. Recall that $X^{(d)}$ denotes the $d$th symmetric power of a curve $X$.

\begin{lemma}\label{lemma35} \cite[Lemma 35]{bourdon2024classification}
    Let the map $f: X \to Y$ be a finite morphism of curves of degree $d$. Then $f$ induces a non-constant morphism $f^* : Y \to X^{(d)}$.
\end{lemma}

In the following lemma, the base change of a curve $X$ to a field $K$ is denoted as $X_K.$
\begin{lemma}\label{descentlemma}\cite[Section 7]{bourdon2024classification}
    Let $X/\Q$ be a curve, $K$ a number field $[K : \Q] = d_2.$ Let $x_0\in X_K$ be a closed point over $K$ of degree $d_1$ be $\mathbb{P}_K^1$-parametrized. Let $x$ be the image of $x_0$ under the map $X_K \to X$. If $K\subseteq \Q(x)$, then $x$ is $\mathbb{P}^1_{\Q}$-parametrized of degree $d = d_1 d_2.$
\end{lemma}

\subsection{Proof of main theorem}

\begin{theorem}\label{thm38}
    Let $E/\Q$ be an elliptic curve with mod $n$ image $G(n)$ of genus $0$, where $n$ is a positive integer. Then every $x \in X_0(n)$ with $j$-invariant $j(x) = j(E)$ is $\mathbb{P}^1$-parameterized.
    \begin{proof}
        Let $x = (E, \langle P\rangle) \in X_0(n)$. Let $U(n)$ be the upper triangular matrix subgroup of $\GL(\Z/n\Z)$:
        \begin{align*}
            U(n) = \left\{\begin{pmatrix}
                a & b \\ c & d
            \end{pmatrix} : c \equiv 0\Mod{n}, ad \in (\Z/n\Z)^{\times}\right\}.
        \end{align*}
        We take $G(n)$ to be written with respect to the basis $\{P, Q\}$ of $E[n]$. Now, let $B = U(n) \cap G(n).$ Observe that the group $B$ is precisely the subgroup of matrices in $G(n)$ which stabilize $\langle P\rangle$.


        Note that we do not necessarily have that $B$ is a subgroup with surjective determinant, so we let $S \coloneqq \det B \leq (\Z/n\Z)^{\times}.$
        Since $B \leq G(n)$ is the subgroup stabilizing $\langle P \rangle$, we have
        \begin{align*}
            S = \{\det(g) \in (\Z/n\Z)^{\times} : g \in G(n), g(\langle P \rangle) = \langle P \rangle\}.
        \end{align*}
        We identify $S$ with a subgroup of $\Gal(\Q(\zeta_n)/\Q)\cong(\Z/n\Z)^{\times}$. Let $K$ denote the fixed field of $S$, and define
        \begin{align*}
            d_2 &\coloneqq [K : \Q] \\
            &= \frac{\#(\Z/n\Z)^{\times}}{\#S}.
        \end{align*}
        In particular, since $K$ is the fixed field of the Galois elements which stabilize $\langle P \rangle$, the field $K$ must be contained in every field $L$ in which $\langle P \rangle$ is stabilized by $\Gal(L/\Q).$ In particular, the subgroup $\langle P \rangle$ is stabilized by $\Gal(\Q(x)/\Q)$, so $K \subseteq \Q(x).$ 

        Next, let $y$ be a $K$-rational point on $X_{G(N), K}$ which corresponds to $E/K$, and with respect to the same basis $\{P, Q\}$ of $E[n].$

        Let $\pi : \GL_2(\Z) \to \GL_2(\Z/n\Z)$ be given by reduction modulo $n$. Next we define $\Gamma(B)$ and $\Gamma(G(n))$ to be given by
        \begin{align*}
            \Gamma(B) &\coloneqq \pi^{-1}(B) \cap \SL_2(\Z),\\
            \Gamma(G(n)) &\coloneqq \pi^{-1}(G(n)) \cap \SL_2(\Z).
        \end{align*}
        Then we define $X_B'$ to be the modular curve $H^*/\Gamma(B)$, which we view as a geometrically integral curve over $K$.
        
        We know by \cite[Section 7.3]{gamma0degrees} that there exists a $K$-rational morphism $f: X_B' \to X_{G(n), K}$ of degree $d_1 = [\pm \Gamma(G(n)): \pm \Gamma(B)]$. For the purpose of this proof, we can simplify this index further to $[\pm \Gamma(G(n)): \pm B] = [\Gamma(G(n)) : B]$, since $P, -P \in \langle P \rangle$. So we proceed to define $d_1$ by the index
        \begin{align*}
            d_1 = \deg (f) = [\Gamma(G(n)) : \Gamma(B)].
        \end{align*}
        Now by Lemma $\ref{lemma35}$, the morphism $f$ induces a non-constant morphism $f^* : X_{G(n), K} \to X_B'^{(d_1)}$, which is a pullback.
        \begin{claim}\label{38claim2}
            Given the rational point $y \in X_{G(n), K}$, the pullback divisor $f^*(y)$ on $X_B' $ is also $\mathbb{P}^1_K$-parameterized.
        
        \begin{proof}[Proof of Claim \ref{38claim2}]
            
        Recall the mod $n$ image $G(n)$ is assumed to be genus $0$, so $X_{G(n), K} \simeq \mathbb{P}^1_K$. Now observe that the pullback $f^*(y) \in X_B'^{(d_1)}$ is itself the image of a point under the non-constant morphism $f^*: \mathbb{P}^1_K \to X_B'^{(d_1)}$. Thus, we obtain from Lemma \ref{lemma36} that the irreducible divisor represented by $f^*(y)$ in $X_B'$ is $\mathbb{P}^1_K$-parameterized.
        \end{proof}
        \end{claim}

        We will now define a map from $X_B'$ to $X_0(n)_K$, and show that the degree of the point $x$ on $X_0(n)_K$ is equal to the degree of its preimage in $X_B'$ via this map.

        Define $g: X_B' \to X_0(n)_K$ corresponding to the inclusion $B \leq U(n)$. Let $x_0 \in X_0(n)_K$ be the closed point such that the map $X_0(n)_K \to X_0(n)_{\Q}$ maps $x_0 \mapsto x$. Let $x' \coloneqq f^*(y) \in X_B'$ be the point on $X_B'$ such that $g(x') = x_0$ and $f(x') = y$. The maps we have are given as follows:
\[\begin{tikzcd}
	& {X_B'} &&& {X_B'^{(d_1)}} \\
	{X_0(n)_K} && {X_{G(n), K} \simeq \mathbb{P}^1_K} & {X_0(n)_K^{(d_1)}} && {X_{G(n), K} \simeq \mathbb{P}^1_K}
	\arrow["g"{description}, from=1-2, to=2-1]
	\arrow["f"{description}, from=1-2, to=2-3]
	\arrow["g^{(d_1)}"{description}, from=1-5, to=2-4]
	\arrow["{f^*}"{description}, from=2-6, to=1-5]
\end{tikzcd}\]

        Recall that $x_0 = g(x')$, so $\deg(x_0) = \deg(g(x'))$. So since we know that $x_0$ is the image via $g$ of a $\mathbb{P}^1_K$-parameterized point $x'$, then to show $x_0\in X_0(n)_K$ is $\mathbb{P}^1_K$-parameterized, it suffices by Lemma \ref{lemma37} to prove that the degree condition $\deg(x') = \deg (g(x'))$ holds.
        
        \begin{claim} \label{38claim3}
        As closed points over $K$, we have the equality
        \begin{align}\label{eqnthm38}
        \deg(x') = \deg (g(x')).
        \end{align}

        \begin{proof}[Proof of Claim \ref{38claim3}]
        We argue that each side of Equation \ref{eqnthm38} is precisely equal to $d_1$.
      
        First, consider the right-hand side of Equation \ref{eqnthm38}, $\deg(g(x'))$. By \cite[Section 7.3]{gamma0degrees},
        \begin{align*}
            d_1 &\coloneqq [\Gamma(G(n)) : \Gamma(B)]
            \\ &= [G(n)_K : B].
        \end{align*}
        We show that the degree of the point $g(x')$ is equal to $d_1$, by the Orbit-Stabilizer Theorem. Under the action of the mod $n$ image of $E/K$ which is denoted $G(n)_K$, the size of the orbit of a point $x_0 \in X_0(n)_K$, denoted $|\orb_{G(n)_K}(x_0)|$, is by definition equal to $\deg(x_0)$. So recalling the stabilizer $\stab_{G(n)_K}(x_0) = B$, we have
        \begin{align*}
            \deg(x_0) &= \frac{|G(n)_K|}{|B|} = d_1,
        \end{align*}
        by definition. And $\deg(g(x')) = \deg(x_0)$, so we must also have $\deg(g(x')) = d_1$, as desired.

        Next, consider the left-hand side of Equation \ref{eqnthm38}. Recall for general algebraic curves $C, D$, the inequality $\deg(a) \geq \deg(h(a))$ holds for $a \in C$, $h: C \to D$, and $h(a) \in D$. Thus, the degree of $x'$ is at most $d_1$. Now, we show that $\deg(x') \geq d_1$. Since we have already shown $\deg(x_0) = d_1$, we know that $\deg(x') \geq \deg (g) \cdot \deg(x_0).$ It thus follows that $\deg(x') \geq d_1.$
        \end{proof}
        \end{claim}
        By Claim \ref{38claim3}, we have that $\deg (x_0) = \deg(x') = \deg(g(x'))$.
        
        Now, to conclude that $x_0 \in X_0(n)_K$ is $\mathbb{P}^1_K$-parametrized, we apply Lemma \ref{lemma37} to the morphism $g$. We showed in Claim \ref{38claim2} that $f^*(y)$ is $\mathbb{P}^1_K$-parameterized, so since $x'$ is defined to be $f^*(y)$, it itself is also $\mathbb{P}^1_K$-parameterized. Hence, Lemma \ref{lemma37} gives us that since $x'$ is $\mathbb{P}^1_K$-parameterized, the point $g(x') = x_0$ is also $\mathbb{P}^1_K$-parameterized. It now follows from Lemma \ref{descentlemma} that $x\in X_0(n)_{\Q}$ is $\mathbb{P}^1_{\Q}$-parametrized.
\end{proof}
\end{theorem}

Moreover, \textit{any} point on the modular curve whose $j$-invariant is equal to an elliptic curve with an adelic image of genus $0$ is $\mathbb{P}^1$-parameterized, as seen in the following corollary to Theorem \ref{thm38}.
\begin{corollary}
    Suppose $E/\Q$ is an elliptic curve with a genus $0$ adelic image $G$. Then for any $m \in \Z^+$, any closed point $x \in X_0(m)$ such that $j(x) = j(E)$ is $\mathbb{P}^1$-parameterized.
    \begin{proof}
        If the given adelic image $G$ is genus $0$, then the group $G(m)$ is also genus 0 for every $m$. The corollary thus follows immediately from Theorem \ref{thm38}.
    \end{proof}
\end{corollary}

\section{Overview of main algorithm}\label{section:overview}
In this section, we give an overview of the algorithm NonIsolated, which is the main procedure for determining whether a non-CM $j$-invariant is isolated. The degree formulas used throughout, and the sub-algorithm needed to complete the set of (level, degree) pairs in Steps \ref{open for} through Steps \ref{close for}, are expanded upon in Sections \ref{section:degreeofmap} and \ref{section:primitive} respectively. 

\begin{breakablealgorithm}\label{algorithm:main}
    \caption{NonIsolated}
    \begin{algorithmic}[1]
        \Require{A non-CM $j$-invariant $j \in \Q$.}
        \Ensure{A finite list \verb|{* <a_1, d_1>, ..., <a_k, d_k> *}| of $($level, degree$)$ pairs such that $j$ is isolated if and only if there exists an isolated point $x \in X_0(a_i)$ of degree $d_i$ with $j(x) = j$ for some pair \verb|<a_i, d_i>| in the list.}
        \State Construct an elliptic curve $E/\Q$ with $j(E) = j.$  
        \State Use Zywina's algorithm \cite{zywina} to compute the adelic image $G$ of $E/\Q$ as a subgroup of $\GL_2(\widehat{\Z})$.
        \State\label{step:reducedlevel}Let $m$ be the product of 2, 3, and all primes $p$ for which the mod $p$ Galois representation of $E$ is not surjective. Find the level $m_0$ of the $m$-adic Galois representation associated to $E$, using \cite[Algorithm 2]{bourdon2024classification}.
        \For{\label{open for}each divisor $a$ of $m_0$}
        \For{each closed point $x$ on $X_0(a)$ corresponding to $E$}
        \State\label{step:primitive}Apply Algorithm \ref{algorithm:primitive} to obtain the primitive point of $x$; i.e., obtain the image of $x$ on $X_0(d_a)$, where $d_a$ is the smallest divisor of $a$ such that, for $f : X_0(a) \to X_0(d_a)$, the condition $\deg(x) = \deg(f) \cdot \deg(f(x))$ holds. The point $f(x)$ will be recorded as a pair \verb|<d_a, d>|, which corresponds to the degree $d$ point $f(x) \in X_0(d_a)$.
        \EndFor
        \EndFor\label{close for}
        \State\label{step:riemannroch}Filter the sequence of $($level, degree$)$ pairs from Step \ref{step:primitive} by only returning $($level, degree$)$ pairs for which $d \leq \genus(X_0(d_a))$.
        \State\label{step:genus0}Filter the sequence of $($level, degree$)$ pairs from Step \ref{step:riemannroch} by only returning $($level, degree$)$ pairs \verb|<a_i, d_i>| for which the mod $a_i$ Galois representation of $E/\Q$ is not genus 0.
        \State \textbf{Return} the sequence of $($level, degree$)$ pairs obtained from Step \ref{step:genus0}.
        \end{algorithmic}
\end{breakablealgorithm}

Note that Steps \ref{open for} to \ref{close for} will result in a sequence \verb|{<a_1, d_1>, ..., <a_j, d_j>}| consisting of $($smallest level, degree$)$ pairs, for all divisors $a$ of $m_0$. This is precisely the set of primitive points of $E$, denoted $\mathcal{P}(E)$, that we define in Section \ref{section:primitive}.

So suppose for some non-CM $j$-invariant $j$, the algorithm returns a non-empty list of 
\begin{center}
    \verb|{<a_1, d_1>, ..., <a_n, d_n>}|
\end{center} pairs. Then $j$ is the image of an isolated point on some $X_0(n)$ via the map to the $j$-line if and only if for some \verb|<a_i, d_i>| in this set, there exists a degree $d_i$ isolated point $x$ on $X_0(a_i)$ such that $j(x) = j$. Consequently, if Algorithm \ref{algorithm:main} returns an empty list \verb|{}|, then $j$ is non-isolated.

\begin{example}
    Let $j = -25/2$ be the given non-CM $j$-invariant. A priori, the curve $X_0(15)$ is known to have a sporadic rational point with $j$-invariant $-25/2$, so we should detect a degree 1 isolated point on $X_0(15)$ in our computation; see \cite[Table 4]{lozano-robledo}.
    
    We first construct an elliptic curve with $j(E) = -25/2$. We consider the elliptic curve over $\Q$ with LMFDB label 50.a3 which is given by the equation $y^2+xy+y=x^3-x-2$.

    The adelic level of $E$ is 120. Taking the complete preimage via $\GL_2(\widehat{\Z}) \to \GL_2(\Z/120\Z)$ of the subgroup of in $\GL_2(\Z/120\Z)$ generated by following matrices, we have a subgroup of $\GL_2(\widehat{\Z})$ which is precisely the adelic image $\im \rho_E.$
    \begin{center}
        $\begin{pmatrix}46 & 105 \\ 45& 1]\end{pmatrix}, \begin{pmatrix}1 & 42 \\ 90 & 61\end{pmatrix}, \begin{pmatrix}1 & 0 \\ 60 & 1\end{pmatrix}, \begin{pmatrix}81 & 10 \\ 40 & 81\end{pmatrix}, \begin{pmatrix}
            31& 90 \\ 15& 91
        \end{pmatrix}, \begin{pmatrix}
            1 & 72 \\ 0 & 1
        \end{pmatrix},$ \\ $\begin{pmatrix}41 & 80 \\ 40 & 81\end{pmatrix}, \begin{pmatrix}61 & 90 \\ 45 & 91\end{pmatrix}, \begin{pmatrix}
            1& 0 \\ 30 & 1
        \end{pmatrix}, \begin{pmatrix}
            73 & 90 \\ 0 & 49
        \end{pmatrix}, \begin{pmatrix}
            1 & 36 \\ 60 & 1
        \end{pmatrix}.$
    \end{center}
    Next, we compute the reduced level of $E$ — the level of the 120-adic Galois representation associated to $E$ — to also be $m_0= 120$. We next compute the set of primitive points of this elliptic curve, $\mathcal{P}(E),$ by Algorithm \ref{algorithm:primitive} which we describe with more detail in Section \ref{section:primitive}:
    \begin{align*}
        <15, 15>, <15, 5>, <15, 3>, <15, 1>, <5, 5>, <5, 1>, <3, 3>, <3, 1>, <1, 1>.
    \end{align*}
    Applying the filters in Steps \ref{step:riemannroch} and \ref{step:genus0} to the set $\mathcal{P}(E)$, we eliminate points which we know can be parameterized by $\mathbb{P}^1$. Note that $\genus (X_0(15)) = 1,$ $\genus (X_0(5)) = 0,$ $\genus (X_0(3)) = 0,$ and $\genus (X_0(1)) = 0.$ So we rule out all points in $\mathcal{P}(E)$, except for the degree 1 point on $X_0(15)$, from being isolated since they fail to satisfy the inequality in Step \ref{step:riemannroch}. The mod $15$ image has genus 1, so the degree 1 point on $X_0(15)$ is also not filtered by Step \ref{step:genus0}.
    
    Thus, from running Algorithm \ref{algorithm:main} on $j = -25/2$, we obtain the list $<15, 1>$. By Theorem \ref{primitive iso iff}, this example confirms that $j = -25/2$ is isolated if and only if a rational point on $X_0(15)$ is isolated. Since there are finitely many rational points on $X_0(15)$, the $j$-invariant $-25/2$ is isolated.
\end{example}

\section{Primitive points}\label{section:primitive}
In this section, we show that the problem of testing whether a rational, non-CM $j$-invariant is isolated can be reduced to testing whether a finite set of points — primitive points — are isolated.

\subsection{Definitions and preliminaries}\label{sec7 prelim}

Throughout, we take $m$ to be the product of 2, 3, and all primes $\ell$ for which $\rho_{E, \ell}$ is non-surjective, and $m \in \Z^+$ by Serre's Open Image Theorem \cite{serre}. For such $m$, we define the reduced level of $E$.
\begin{definition}
     Let $m \in \Z^+$ be the product of 2, 3, and all primes $\ell$ for which the mod $\ell$ Galois representation of $E$ is not surjective. The \textbf{reduced level of $E$} is the level of the $m$-adic Galois representation associated to $E.$
\end{definition}

The following theorem of Menendez is an analogue of \cite[Theorem 4.1]{belov} for the modular curve $X_0(n),$ stated for the reduced level $m_0$ of $E$. See a general statement and proof in Appendix \ref{appendix:Menendez}. Recall that $\Supp(m_0)$ denotes the set of prime divisors of $m_0.$
\begin{theorem}\cite[Theorem 5.1]{Zonia}\label{Zonia}
    Let $j$ be a non-CM $j$-invariant. Let $k \coloneqq \Q(j)$ and fix an elliptic curve $E/k$ with $j(E) = j$. Let $S_E$ denote the set
    \begin{align*}
    S_E \coloneqq \{2, 3\} \cup \{ \ell : \rho_{E, \ell^{\infty}} (\Gal(\overline{\Q}/\Q)) \neq \GL_2(\Z_{\ell})\}.
    \end{align*}
    Define
    \begin{align*}
        \mathfrak{m}_S \coloneqq \prod_{\ell \in S_E}\ell.
    \end{align*}
    Let $m_0 \in \Z^+$ denote the reduced level of $E$. If $x \in X_0(n)$ is a closed point with $j(x) = j(E)$, then $\deg(x) = \deg(f) \cdot \deg(f(x))$, where $f$ denotes the natural map $X_0(n) \to X_0(\gcd(n, m_0))$.
\end{theorem}

     
     Consequently, if there exists an isolated point $x \in X_0(n)$, for any $n \in \Z^+$, then $f(x) \in X_0(\gcd(n, m_0))$ is isolated by Theorem \ref{belovdeg}, where $f: X_0(n) \to X_0(\gcd(n, m_0))$ denotes the natural map.

\subsection{Primitive points}\label{subsec:primitivepts}

Let $E$ be an elliptic curve over a number field $k$, such that $j = j(E)$ is a non-CM $j$-invariant. By Section \ref{sec7 prelim}, each isolated point $x \in X_0(n)$ with $j(x) = j$ must map to an isolated point on some $X_0(a)$, where $a$ is a divisor of the reduced level $m_0$ of $E$. We use this fact to characterize the finite set of $($level, degree$)$ pairs that will be sufficient to test for isolated points. First, we define the primitive point of $x \in X_0(n)$.

\begin{definition}\label{defn:myprimitivepoints}
    Let $E$ be an elliptic curve over a number field $k$, such that $j = j(E)$ is a non-CM $j$-invariant. Let $a$ be some divisor of the reduced level $m_0$, and let $x\in X_0(a)$ be a closed point. Suppose $d_a$ is the smallest divisor of $a$ (which could be equal to $a$ itself) such that, if $f: X_0(a) \to X_0(d_a)$ is the natural map, then the degree condition
    \begin{align*}
        \deg(x) = \deg(f) \cdot \deg(f(x))
    \end{align*}
    holds. We call the point $f(x) \in X_0(d_a)$ the \textbf{primitive point of $x$}. Denote the primitive point of $x \in X_0(a)$ as $\mathcal{P}(x)$.
\end{definition}

\begin{theorem}\label{primitive iso iff}
    Let $E/\Q$ be a non-CM elliptic curve, and let $m_0 \in \Z^+$ be the reduced level of $E$. Define $\mathcal{P}(E)$ to be given by the set of all $\mathcal{P}(x)$ for $x \in X_0(a)$ with $a \mid m_0$. Then the rational number $j(E)$ is isolated if and only if there exists an isolated point in the set $\mathcal{P}(E)$.
\end{theorem}

\begin{proof}[Proof of Theorem \ref{primitive iso iff}]
    Let $E/\Q$ be a non-CM elliptic curve. If there exists an isolated point in $\mathcal{P}(E)$, then by definition $j(E)$ is isolated.

    Conversely, suppose $j(E)$ is isolated. Then there exists an isolated point $x \in X_0(n)$ with $j(x) = j(E)$, for some $n \in \Z^+$. So by Theorem \ref{Zonia}, we know that $\deg(x) = \deg(f) \cdot \deg(f(x))$ holds, where $f: X_0(n) \to X_0(\gcd(n, m_0))$ is the natural map. By Theorem \ref{belovdeg}, the point $f(x) \in X_0(\gcd(n, m_0))$ is isolated.

    Now, let $a \coloneqq \gcd(n, m_0)$, and take $d_a$ to be defined as in Definition \ref{defn:myprimitivepoints}, where $d_a$ is the smallest divisor of $\gcd(a)$ such that the degree condition holds; i.e., if $g : X_0(a) \to X_0(d_a)$ is the natural map, then the point $g(f(x)) \in X_0(d_a)$ is isolated by Theorem \ref{belovdeg}.
    
    The point $g(f(x)) \in X_0(d_a)$ is, by construction, in the set of primitive points $\mathcal{P}(E)$, and is isolated, as desired.
\end{proof}

\subsection{Connecting Definition \ref{defn:myprimitivepoints} to existing literature}\label{subsec:primitiveptsconnections}
In this subsection, we give another construction for the set $\mathcal{P}(E)$, analogous to the definition given in \cite[Section 5]{bourdon2024classification} with respect to the curve $X_1(n)$. In Section \ref{subsubsection:prelim}, we provide preliminary results, before proving in Section \ref{subsubsection:equivalence} that our definition of primitive points, Definition \ref{defn:myprimitivepoints}, is equivalent to the analogous definition in \cite{bourdon2024classification}.

Let $E/\Q$ be a non-CM elliptic curve, and let $m \geq 1$. We define the directed graph $G(E, m)$ as follows. First consider the set of closed points on $X_0(n)$, for all $n \mid m$, which correspond to the given elliptic curve $E$. We take \textbf{vertices} of $G(E, m)$ to be tuples $(x, n, d)$ such that:
\begin{itemize}
    \item[(i)] $n \mid m$,
    \item[(ii)] $x$ is a closed point on $X_0(n)$ of degree $d$, and
    \item[(iii)] $j(x) = j(E)$.
\end{itemize}
We connect the vertices $(x, n, d)$, $(x', n', d')$ with a \textbf{directed edge} from $(x, n, d)$ to $(x', n', d')$, if $f: X_0(n) \to X_0(n')$ denotes the natural map and the following conditions are true:
\begin{itemize}
    \item[(i)] $n' \mid n$ and $n' \neq n$,
    \item[(ii)] $x' = f(x)$,
    \item[(iii)] $d = d' \deg(f).$
\end{itemize}
The graph $G(E, m)$ is a directed acyclic graph.

\begin{proposition}
    Let $E/\Q$ be a non-CM elliptic curve. The graph $G(E, m)$ is transitive. That is, if $n_2| n_1|n$ are proper divisors and $(x, n, d)$, $(x_1, n_1, d_1)$, and $(x_2, n_2, d_2)$ are vertices of $G(E, m)$, then there is a directed edge from $(x, n, d)$ to $(x_2, n_2, d_2)$ if and only if there are directed edges from $(x, n, d)$ to $(x_1, n_1, d_1)$ and $(x_1, n_1, d_1)$ to $(x_2, n_2, d_2)$.
    \begin{proof}
    Let $f_1 : X_0(n) \to X_0(n_1)$ and $f_2 : X_0(n_1) \to X_0(n_2)$ denote the natural projection maps.
    
    First, suppose that there is an edge from $(x, n, d)$ to $(x_2, n_2, d_2)$. We want to show that $d = d_1 \cdot \deg(f_1)$. Suppose toward contradiction that $d < d_1 \cdot  \deg(f_1).$ Since $d_1 \leq d_2 \cdot \deg(f_2)$, we know that
    \begin{align*}
        d < d_2 \deg(f_1) \cdot \deg(f_2). 
    \end{align*}
    We have contradicted the assumption that $d = d_2 \cdot \deg(f_2)$.

    Conversely, suppose that there are directed edges from $(x, n, d)$ to $(x_1, n_1, d_1)$ and $(x_1, n_1, d_1)$ to $(x_2, n_2, d_2)$. We thus know that $d = d_1 \deg(f_1)$ and $d_1 = d_2 \deg(f_2)$, so
    \begin{align*}
        d &= d_2\deg(f_2)\cdot\deg(f_1).
    \end{align*}
    Thus, there is a directed edge from $(x, n, d)$ to $(x_2, n_2, d_2)$.
    \end{proof}
\end{proposition}

\begin{remark}
    Recall that for a fixed vertex $(x, n, d) \in G(E, m)$, its \textbf{descendants} are defined to be all vertices $(x', n', d')$ which can be reached via a path from $(x, n, d)$ to $(x', n', d')$. The vertex $(x, n, d)$ and its descendants induce a subgraph of $G(E, m)$ which is directed, acyclic, and has a single source (namely, $(x, n, d)$).
\end{remark}

\begin{definition}\label{defn:algorithmsprimitivepoints}
    Define the set $S$ to be the union of sinks of the graphs $G(E, m)$, for all $m \in \Z^+.$ The set $S$ is, in particular, analogous to the set of primitive points given in \cite[Section 5]{bourdon2024classification}. 
\end{definition}

\subsubsection{Preliminary results}\label{subsubsection:prelim}

We establish the following theorem, which is concerning the residue field of points on the modular curve $X_0(n)$, which we use in the proof of Theorem \ref{thm:degreegcd}.

\begin{theorem}\label{theorem:residuefieldsequal}
    Let $n_1, n_2 \in \Z^+$ and $n = \lcm(n_1, n_2)$, and let $g = \gcd(n_1, n_2)$. Suppose $x = (E, \langle P \rangle) \in X_0(n)$ for an elliptic curve $E$ with $j(E) \not\in \{0, 1728\}$. Next, define $x_1 = (E, \langle \frac{n_2}{g} P \rangle) \in X_0(n_1)$ and $x_2 = (E, \langle \frac{n_1}{g}P\rangle) \in X_0(n_2)$. Then the residue field $\Q(x)$ is equal to the composite field $\Q(x_1) \Q(x_2).$
    \begin{proof}
        First, we show that $\Q(x) \subseteq \Q(x_1)\Q(x_2)$. Note that by the Fundamental Theorem of Galois Theory, it suffices to show an element $\sigma$ that is both in $\Gal(\overline{\Q}/\Q(\langle \frac{n_2}{g}P\rangle))$ and 
        $\Gal(\overline{\Q}/\Q(\langle \frac{n_2}{g}P\rangle))$ will fix $\langle P \rangle$; that is, we have $\rho_{E, \frac{n_2}{g}}$ and $ \rho_{E, \frac{n_1}{g}}$ are upper triangular. In particular, since $\rho_{E, \frac{n_1}{g}}$ is upper triangular, the image of $\sigma$ is also upper triangular mod $n_1.$
        
        Now, note that lcm$(n_1, n_2) = \frac{n_2}{g}n_1$, so it follows that since $\sigma$ is upper triangular mod $n_1$ and mod $\frac{n_2}{g}$, it is also upper triangular mod lcm$(n_1, n_2).$ Thus $\langle P \rangle$ is fixed by the action of $\Gal(\overline{\Q}/\Q(x_1)\Q(x_2))$, and so $x$ is defined over the composite field $\Q(x_1)\Q(x_2)$.

        Since we also have that $\Q(x_i) \subseteq \Q(x)$ for each $i \in \{1, 2\}$, it follows that the compositum $\Q(x_1)\Q(x_2)$ is also contained in $\Q(x)$.
    \end{proof}
\end{theorem}

We show that given a vertex $(x, n, d)$ with an outgoing edge to two distinct vertices $(x_1, n_1, d_1)$ and $(x_2, n_2, d_2)$, there must exist a unique third vertex $(x_3, \gcd(n_1, n_2), d_3)$ such that they both connect to via a directed edge. In Theorem \ref{thm:degreegcd}, we show the existence of such a vertex, and in Theorem \ref{thm:graphuniqueness} we show its uniqueness.

\begin{theorem}\label{thm:degreegcd}
    Let $E/\Q$ be a non-CM elliptic curve. Let $x = (E, \langle P \rangle) \in X_0(n)$ be a closed point. Let $n_1, n_2 $ be divisors of $n$ such that $n = \lcm(n_1, n_2)$, and denote $f_1 : X_0(n) \to X_0(n_1)$ and $f_2 : X_0(n) \to X_0(n_2)$. Suppose the degree conditions
    \begin{align*}
        \deg(x) &= \deg(f_1) \cdot \deg (f_1(x)),\\
        \deg(x) &= \deg(f_2) \cdot \deg (f_2(x))
    \end{align*}
    hold. Then we must have $\deg(x) = \deg(f)\cdot \deg(f(x))$, where $f : X_0(n) \to X_0(\gcd(n_1, n_2))$.
    \begin{proof}
        Note that if $\gcd(n_1, n_2) = n_1$ or $n_2$ in Theorem \ref{thm:degreegcd}, then the theorem immediately holds.
    
        Let $n_1, n_2$ be distinct, proper divisors of $n$. Then by the given degree conditions, we have vertices $(x, n, \deg(x))$, $(x_1, n_1, \deg(x_1))$, and $(x_2, n_2, \deg(x_2))$ in the graph $G(E, m)$, with directed edges from $(x, n, \deg(x))$ to $(x_i, n_i, \deg(x_i))$ for each $i \in \{1, 2\}$. We want to show that there is a directed edge \begin{align*}(x_i, n_i, \deg(x_i)) \to (f(x), \gcd(n_1, n_2), \deg(f(x))),\end{align*} where $i \in \{1, 2\}$. Denote $g_i : X_0(n_i) \to X_0(\gcd(n_1, n_2))$, where $i \in \{1, 2\}.$ We have the following maps.
\[\begin{tikzcd}
	& {X_0(n)} \\
	{X_0(n_1)} && {X_0(n_2)} \\
	& {X_0(\gcd(n_1, n_2))}
	\arrow["{f_1}"', from=1-2, to=2-1]
	\arrow["{f_2}", from=1-2, to=2-3]
	\arrow["f", from=1-2, to=3-2]
	\arrow["{g_1}"', from=2-1, to=3-2]
	\arrow["{g_2}", from=2-3, to=3-2]
\end{tikzcd}\]

        Since $g \coloneqq \gcd(n_1, n_2)$ is a proper divisor of $n_1, n_2$, there exists $k_1, k_2$ such that $n_1 = g \cdot k_1$ and $n_2 = g \cdot k_2.$
            
By Theorem \ref{theorem:residuefieldsequal}, we know that for $x_1 \in X_0(n_1),$ $x_2 \in X_0(n_2),$ we have $\Q(x) = \Q(x_1)\Q(x_2).$ Thus, $\deg(f_1) = [\Q(x_1)\Q(x_2) : \Q(x_1)],$ and $\deg(f_2) = [\Q(x_1)\Q(x_2) : \Q(x_2)].$ Next observe that we now have the following field extensions.
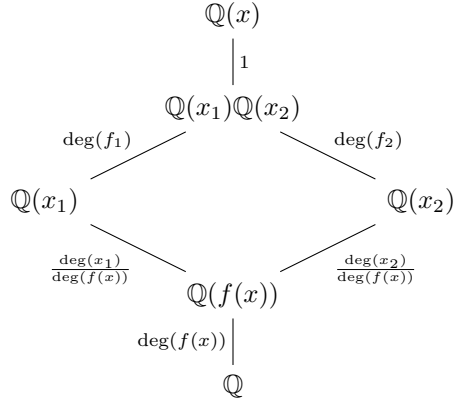
\[\begin{tikzcd}
	& {\mathbb{Q}(x)} \\
	& {\mathbb{Q}(x_1)\mathbb{Q}(x_2)} \\
	{\mathbb{Q}(x_1)} && {\mathbb{Q}(x_2)} \\
	& {\mathbb{Q}(f(x))} \\
	& {\mathbb{Q}}
	\arrow["1", no head, from=1-2, to=2-2]
	\arrow["{\deg(f_1)}"', no head, from=2-2, to=3-1]
	\arrow["{\deg(f_2)}", no head, from=2-2, to=3-3]
	\arrow["{\frac{\deg(x_1)}{\deg(f(x))}}"', no head, from=3-1, to=4-2]
	\arrow["{\frac{\deg(x_2)}{\deg(f(x))}}", no head, from=3-3, to=4-2]
	\arrow["{\deg(f(x))}"', no head, from=4-2, to=5-2]
\end{tikzcd}\]
    Note for $p \mid k_1$, we have $p \nmid g$ if and only if $p \nmid n_2$. Similarly, for $p \mid k_2$, we have $p \nmid g$ if and only if $p \nmid n_1$. Also, we have $n = n_2 \cdot k_1 = n_1 \cdot k_2$. So, we have $\deg(f_2) = \deg(g_1)$ and $\deg(f_1) = \deg(g_2)$, which gives the following lower bounds:
    \begin{align*}
        \frac{\deg(x_2)}{\deg(f(x))} &\leq \deg(f_1), \\
        \frac{\deg(x_1)}{\deg(f(x))} &\leq \deg(f_2).
    \end{align*}
    
    From the above diagram, we also know by properties of composite fields that
    \begin{align*}
        \deg(f_1) &\leq \frac{\deg(x_2)}{\deg(f(x))},\\
        \deg(f_2) &\leq \frac{\deg(x_1)}{\deg(f(x))}.
    \end{align*}
    
    It follows that
    \begin{align*}
        \deg(x_1) &= \deg(f(x)) \cdot \deg(g_1), \\ \deg(x_2) &= \deg(f(x)) \cdot \deg(g_2).
    \end{align*}
    Thus, we have that $\deg(x) = \deg(f(x)) \cdot \deg(f_2) \cdot \deg(g_2) = \deg(f(x)) \cdot \deg(f).$
That is, $\deg(x) = \deg(f) \cdot \deg(f(x))$, as desired.
\end{proof}
\end{theorem}

\begin{theorem}\label{thm:graphuniqueness}
    Let $E/\Q$ be a non-CM elliptic curve. For any vertex $(x, n, d)$, the subgraph induced by $(x, n, d)$ and its descendants has a unique sink.
    \begin{proof}
           Fix a vertex $(x, n, d)$ of the graph $G(E, m)$. Suppose the induced subgraph of its descendants has two distinct sinks, $(x_1, n_1, d_1)$ and $(x_2, n_2, d_2).$

           Let $g \coloneqq \gcd(n_1, n_2).$ Since $(x_1, n_1, d_1)$ and $(x_2, n_2, d_2)$ are both sinks, there is no outgoing edge from either vertex. Observe that we cannot not have that $g$ is a proper divisor of $n_1$ or $n_2$. If $g$ were a proper divisor of both $n_1, n_2$, then there would be an outgoing edge from both $(x_1, n_1, d_1)$ and $(x_2, n_2, d_2)$ to $(x_3, \gcd(n_1, n_2), d_3),$ by Theorem \ref{thm:degreegcd}.
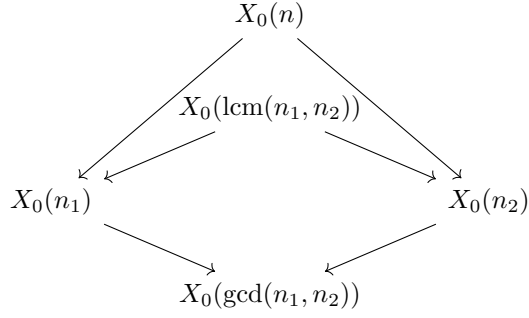
\[\begin{tikzcd}
	& {X_0(n)} \\
	& {X_0(\text{lcm}(n_1, n_2))} \\
	{X_0(n_1)} && {X_0(n_2)} \\
	& {X_0(\gcd(n_1, n_2))}
	\arrow[from=1-2, to=3-1]
	\arrow[from=1-2, to=3-3]
	\arrow[from=2-2, to=3-1]
	\arrow[from=2-2, to=3-3]
	\arrow[from=3-1, to=4-2]
	\arrow[from=3-3, to=4-2]
\end{tikzcd}\]
           
           So, we must have $g = n_1$ or $g = n_2.$ Suppose that $n_1 = g.$ Then, since $g | n_2$, we have that there is a directed edge from $(x_2, n_2, d_2)$ to $(x_1, n_1, d_1)$.
           
           This contradicts the assumption that $(x_2, n_2, d_2)$ is a sink; thus, there is a unique sink of the induced subgraph of descendants of $(x, n, d).$
    \end{proof}
\end{theorem}

\subsubsection{Proof of equivalence of definitions}\label{subsubsection:equivalence}

Let $m_0$ denote the reduced level of an elliptic curve $E$.

Recall that by Theorem \ref{Zonia}, we have that if $x \in X_0(n)$ and $f : X_0(n) \to X_0(\gcd(n, m_0))$, then
    \begin{align*}
        \deg(x) &= \deg(f) \cdot \deg(f(x)).
    \end{align*}
    It follows that if $v$ is a sink of $G(E, n)$ for some $n \in \Z^+$, then $n$ must be, in particular, a divisor of $m_0$. Thus, the set of sinks of $G(E, m)$ for $m \geq 1$ is, in particular, the set of sinks of $G(E, m_0).$

\begin{theorem}\label{theorem:defnequivalence}
    The set of points on some $X_0(n)$ corresponding to a sink in $S$ from Definition \ref{defn:algorithmsprimitivepoints} is equal to the set of primitive points given in Definition \ref{defn:myprimitivepoints}.
    \begin{proof}
        Let $E/\Q$ be a non-CM elliptic curve, and denote its reduced level by $m_0$.
        
        First, we show that if a point $x\in X_0(n)$ corresponds to a sink on the graph $G(E, m)$, then $x \in X_0(n)$ is in the set of primitive points $\mathcal{P}(E)$ from Definition \ref{defn:myprimitivepoints}. Let $x \in X_0(n)$ be a a sink in the set $S$. That is, the vertex $(x, n, d)$ is a sink of $G(E, m)$, for some $m \in \Z^+$.
        
        First note that $n|m_0$, and observe that $\gcd(n, m_0) = n$. We could not have $\gcd(n, m_0)$ be a proper divisor $n$. If it were, there would be an outgoing edge $(x, n, d) \to (x', \gcd(n, m_0), \deg(x')),$ where $x'$ is the image of $x$ on $X_0(\gcd(n, m_0))$ under the natural projection map, by Theorem \ref{thm:degreegcd}. This contradicts the fact that the vertex $(x, n, d)$ is a sink of $G(E, m)$.
        
        In particular, the vertex $v$ is a sink of $G(E, m_0).$ Thus, there is no proper divisor $a$ of $n$ for which it is true that $\deg(x) = \deg(f)\cdot\deg(f(x))$, where $f : X_0(n) \to X_0(a)$ denotes the natural map. So $x \in X_0(n)$ belongs to the set of primitive points $\mathcal{P}(E)$ from Definition \ref{defn:myprimitivepoints}.

        Next, we show that if $x \in X_0(d_n)$ is a primitive point, then we must have that $x$ corresponds to a sink in the set $S$ from Definition \ref{defn:algorithmsprimitivepoints}. Let $x \in X_0(d_n)$ be a primitive point, where $d_n | m_0.$
        
        It follows that $(x, d_n, d)$ must be a sink of the graph $G(E, n)$. Toward contradiction, suppose there were an outgoing edge connecting to a vertex $(x', n', d')$, with $n'$ a proper divisor of $d_n$. In particular, $n'$ divides $n.$ Then we immediately contradict the fact that $d_n$ is the \textit{smallest} divisor of $n$ for which the degree condition holds.
    \end{proof}
\end{theorem}

\section{Algorithm for computing primitive points}\label{section:primitivealgorithm}
Let $E/\Q$ be a non-CM elliptic curve, and let $m \in \Z^+$. In this section, we reduce the problem of testing whether $j(E) \in \Q$ is non-isolated to testing whether a finite set of $($level, degree$)$ pairs all correspond to points which are non-isolated. See Theorem \ref{primitive iso iff}.

Recall that the reduced level is the level of the $m$-adic Galois representation associated to $E$, where $m \in \Z^+$ is the product of $2, 3$, and all primes $p$ for which the mod $p$ Galois representation of $E$ is not surjective. Throughout, we will use $m_0$ to denote the reduced level.

\subsection{Algorithm for primitive points}

We describe an algorithm for computing the set of primitive points associated to a non-CM elliptic curve. In Algorithm \ref{algorithm:primitive}, we compute the set of primitive points associated to each $x \in X_0(n)$, for each divisor $n$ of $m_0$. This process returns a sequence
\begin{center}
    \verb|{<a_1, d_1>, ..., <a_k, d_k>}|.
\end{center}

\begin{breakablealgorithm}   
\label{algorithm:primitive}\caption{PrimitiveDegreesOfPoints}
    \begin{algorithmic}[1]
        \Input $G \leq \GL_2(\Z/m_0\Z)$, where $G \coloneqq \im \rho_{E, m_0}$
        \Output{A sequence of primitive level-degree pairs \verb|<a_1, d_1>, ..., <a_n, d_n>| corresponding to closed points $x \in X_0(a_i)$.}
        \For{each divisor $a$ of $m_0$}
        \State Compute $C$, the set of all cyclic submodules $v$ of order $a$ of $(\Z/a \Z)^2$. Each submodule is a possible subgroup of order $a$ corresponding to a point on $X_0(a_i),$ where $a_i$ divides $a.$
        \For {each $v \in C$}
        \State Let $H$ denote the subgroup of $\GL_2(\Z/a\Z)$ given by $G$ mod $a.$
        \State Compute the list of divisors of $a$, given by $a_1, \dots, a_k$ ordered such that $a_1 \leq \dots \leq a_k$.
        \For{ each divisor $a_i$ of $a$}\label{step:divisorlist}
        \State\label{Step:degreexi}In this step, we compute $\deg(f(x))$, where $f : X_0(a) \to X_0(a_i)$ is the natural projection map and $f(x) \in X_0(a_i).$ Define $H_{i} \leq \GL_2(\Z/a_i\Z)$ to be the subgroup $H$ modulo $a_i.$ Also let $L \leq (\Z/a_i\Z)^2$ be generated by the generator of $v$ modulo $a_i$. The degree of $f(x)$ is given by $|\orb_{H_i}(L)|$.
        \State\label{Step:degequal}Now, we check if $\deg(x) = \deg(f) \cdot \deg(f(x)).$ Check if
        \begin{align*}
            | \orb_H(v) | &= \deg(f) \cdot |\orb_{H_i} (L)|.
        \end{align*}
        \State \textit{Case 1:} If we have equality in Step $\ref{Step:degequal}$, then we have found that $a_i$ is the smallest divisor for which the degree condition holds (denoted $d_a$ in Definition \ref{defn:myprimitivepoints}), because our list of divisors from Step \ref{step:divisorlist} is ordered from smallest to largest. Thus, we record the level-degree pair, ($a_i, |\orb_{H_i}(L)|)$, as a primitive point. End the for-loop through the list of divisors of $a_i$ of $a$.
        \State \textit{Case 2:} If equality in Step \ref{Step:degequal} did not hold and $a_i < a$, then we continue to the next divisor in Step \ref{step:divisorlist}. \State \textit{Case 3:} Otherwise, if $a_i = a$, then we must have found that equality in Step \ref{step:divisorlist} did not hold for any proper divisors $a_i < a$. In this case, the point $x \in X_0(a)$ is its own primitive point, and so we record the level-degree pair $(a, |\orb_H(v)|)$ as a primitive point.
        \EndFor
        \EndFor
        \EndFor
    \end{algorithmic}
\end{breakablealgorithm}

To test whether $\deg (x) = \deg (f)\cdot\deg f(x)$ in Step \ref{Step:degequal}, we use the degree computations for the natural projection map $f$ from Section \ref{section:degreeofmap}.

\subsection{Worked example}

\begin{example}
    Let $j = 43307231/82944$, a non-CM $j$-invariant. We construct an elliptic curve $E$ with the given $j$-invariant. We can consider the elliptic curve over $\Q$ with LMFDB label 726.a1, which is given by the equation $y^2+xy=x^3+x^2+21657x-1855179.$
    
    The elliptic curve $E$ has adelic level and reduced level both equal to 4. To compute the set $\mathcal{P}(E)$, we consider the 3 divisors of the reduced level: 1, 2, and 4. So applying our algorithm, we check the closed points above $j = 43307231/82944$ of $X_0(2)$ and $X_0(4)$. Using computations from Section \ref{section:degreeofmap}, the following degrees of the natural projection maps are given as follows.

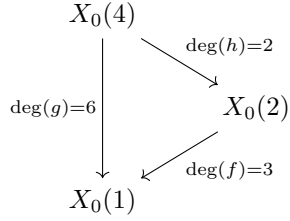
\[\begin{tikzcd}
	{X_0(4)} \\
	& {X_0(2)} \\
	{X_0(1)}
	\arrow["{\deg(h)=2}", from=1-1, to=2-2]
	\arrow["{\deg(g) = 6}"', from=1-1, to=3-1]
	\arrow["{\deg(f) = 3}", from=2-2, to=3-1]
\end{tikzcd}\]
    
    We first begin with closed points on $X_0(2)$. There is one closed point of degree $3$ above $j = 43307231/82944$ on $X_0(2)$. Since the only divisor of $2$ strictly smaller than 2 is 1, we only consider the degree of $f(x)$ where $f$ is given by the natural projection map from $X_0(2) \to X_0(1)$. The map $f: X_0(2) \to X_0(1)$ is of degree 3, so we have that for each of these points, the degree condition $\deg(x) = \deg(f) \cdot \deg(f(x))$ holds when mapping to a degree 1 point on $X_0(1)$. Thus, the primitive point of the closed degree 3 point on $X_0(2)$ is represented by the (level, degree) pair $<$1, 1$>$.

    Next, we consider the closed points above above $j = 43307231/82944$ on $X_0(4)$, all of which are degree 3. Note that for $X_0(4)$, there are two maps to consider: the map $g: X_0(4) \to X_0(1)$, and the map $h: X_0(4) \to X_0(2),$ where $\deg(g) = 6$ and $\deg(h) = 2$. The primitive points are defined such that we take the image of the point on the modular curve with the smallest level, so we check the degree condition on $g$ before $h$. Observe that $\deg(x) \neq \deg(g)\cdot\deg(g(x))$, so $g$ cannot give rise to a primitive point.

    Next, observe that $\deg(x) \neq \deg(h)\cdot \deg(h(x))$, so $h$ also cannot give rise to a primitive point. Hence, the degree 3 points on $X_0(4)$ are their own primitive points.
    
    So the resulting set $\mathcal{P}(E)$, recorded as (level, degree) pairs, is given by $<1, 1>$ and $<4, 3>$.
\end{example}

\section{Validity of main algorithm}\label{section:validity}
In this section, we prove the validity of our main algorithm, Algorithm $\ref{algorithm:main}$.

\begin{theorem}\label{validity}
    Let $j \in \Q$ be a non-CM $j$-invariant. If Algorithm \ref{algorithm:main} returns
\begin{center}
    \verb|<a_1, d_1>, ..., <a_k, d_k>|,
\end{center}
then any isolated point $x \in X_0(n)$ for $n \in \Z^+$ with $j(x) = j$ maps, via the natural map, to an isolated point of degree $d_i$ on $X_0(a_i)$ for some $1 \leq i \leq k$.

\begin{proof}
    Let $E/\Q$ be an elliptic curve with $j(E) = j$. Note that the degrees of closed points do not depend on choice of an equation, so the choice of $E$ is arbitrary.

    Let $a \in \Z^+$ denote the level of the adelic Galois representation of $E$. We can obtain the image of the adelic Galois representation $G = \im \rho_E$ via Zywina's algorithm \cite{zywina}, and the output of this computation is represented by $G(a)$, where $G(a) \leq \GL_2(\Z/a\Z).$
    As proven in Proposition 22 of \cite{bourdon2024classification}, we can use the adelic image $G(a)$ to now compute the level of the $m$-adic Galois representation that is associated to $E$, where $m$ is the product of $2, 3$, and all primes $p$ for which the mod $p$ Galois representation of $E$ is not surjective. We let $m_0$ denote the level of this $m$-adic Galois representation.

    Recall that by Theorem \ref{primitive iso iff}, there exists a closed point $x \in X_0(d)$ for some $d \in \Z^+$ which is isolated if and only if there exists an isolated point corresponding to a point in the set $\mathcal{P}(E)$. So we compute the set of (level, degree) pairs corresponding to points in $\mathcal{P}(E)$, by applying Algorithm \ref{algorithm:primitive} to $G = G(m_0)$ and each divisor $a$ of $m_0$. For each closed point on $X_0(a)$, we record its associated primitive point $\mathcal{P}(x)$ on $X_0(d_a)$, as defined in Section \ref{subsec:primitivepts}, thus obtaining the set
    \verb|{<a_1, d_1>, ..., <a_m, d_m>}|.
    
    In the next set of filters, we eliminate (level, degree) pairs corresponding to points in $\mathcal{P}(E)$ which we know must correspond to parameterized points.
    
    First, we rule out points from being isolated via the Riemann-Roch Theorem. Pairs $(a_i, d_i)$ for which $d_i$ is strictly greater than $\genus(X_0(a_i))$ are non-isolated, by Lemma \ref{thm:riemannrochlemma}. So now consider points $x$ such that $\deg(x) \leq g(X_0(n))$.

    We next remove any (level, degree) pairs $(a_i, d_i)$ from our list for which the mod $a_i$ Galois representation $\im \rho_{E, a_i}$ is genus 0, which are not $\mathbb{P}^1$-isolated by Theorem \ref{thm38}. Denote the resulting list of $($level, degree$)$ pairs \verb|{<a_1, d_1>, ..., <a_k, d_k>}|.
    
    Any isolated point $x \in X_0(n)$ for any $n \in \Z^+$ with $j(x) = j$ maps via the natural projection map to an isolated point of degree $d_i$ on $X_0(a_i)$ for some $i \in \{1, \dots, k\}$.
\end{proof}
\end{theorem}

\begin{corollary}\label{corollaryempty}
    Let $j\in \Q$ be a non-CM $j$-invariant. If applying Algorithm \ref{algorithm:main} to $j$ yields an empty list \verb|{}|, then $j$ is not the image of an isolated point on $X_0(n)$ for any $n\in \Z^+$; that is, $j$ is not isolated.
\end{corollary}

\section{An isolated $j$-invariant for $X_1(n)$ that is not isolated for $X_0(n)$}\label{section:351/4}
Let $E$ be the elliptic curve over $\Q$ with LMFDB label 338.c2, which is given by the equation $y^2 + xy = x^3 - x^2 + x + 1,$ and has $j$-invariant $\frac{351}{4}$.

The adelic level of $E$ is 364, and the reduced level of $E$ is 28. Applying Algorithm \ref{algorithm:primitive}, we obtain the following list of (level, degree) pairs corresponding to the primitive points of $E$:
\begin{align*}
    <28, 21>, <4, 3>, <7, 7>, <1, 1>, <28, 3>, <7, 1>.
\end{align*}
Now, we apply both filters to this set. Note that $\genus(X_0(28)) = 2,$ $\genus(X_0(4)) = 0,$ $\genus(X_0(7)) = 0,$ and $\genus(X_0(1)) = 0.$
Lemma \ref{thm:riemannrochlemma} gives us that all points corresponding to this set of (level, degree) pairs are $\mathbb{P}^1$-parameterized. Hence, Algorithm \ref{algorithm:main} returns an empty list which shows that $j = 351/4$ is non-isolated for $X_0(n)$.

Notably, this $j = 351/4$ corresponds to an isolated point of degree 9 on $X_1(28)$, by \cite[Theorem 2]{odddegree}. Also recall that it was shown by \cite{bourdon2024classification} that all of the 14 exceptional $j$-invariants (which do give isolated points for $X_0(n)$) are non-isolated for $X_1(n)$. Therefore, we can conclude the following.

\begin{theorem}\label{subsets result revisited}
    The set of rational $j$-invariants corresponding to isolated points on $X_1(n)$ is neither a subset nor a superset of those corresponding to isolated points on $X_0(n)$.
\end{theorem}

\begin{appendices}
\section{On work of Menendez, by Abbey Bourdon}\label{appendix:Menendez}
The results of this section appear in the PhD thesis of Zonia Menendez \cite{Zonia}. Since the resulting paper is not yet available on arXiv and the notation and conventions of \cite{Zonia} are somewhat inconsistent with the present work, we include a short proof of the results needed for ease of reference.

For a non-CM elliptic curve $E$ over a number field $k$, define
\[
S_E \coloneqq \{2,3\} \cup \{ \ell : \rho_{E,\ell^{\infty}}(\Gal_{k}) \not\supset \SL_2(\mathbb{Z}_{\ell})\} \cup \{5, \text{ if } \rho_{E,5^{\infty}}(\Gal_{k}) \neq \GL_2(\mathbb{Z}_{5})\}.
\]

\begin{theorem}[Menendez, \cite{Zonia}]\label{MainThm}
Fix a non-CM elliptic curve $E$ defined over $k=\Q(j(E))$. Let $S=S_{E}$, and let $\mathfrak{m}_S \coloneqq \prod_{\ell \in S} \ell$. Let $M$ be a positive integer with $\Supp(M) \subset S$ satisfying
\[
\im \rho_{E, \mathfrak{m}_S^{\infty}}=\pi^{-1}(\im \rho_{E,M}).
\]
If $x \in X_0(n)$ is a closed point with $j(x)=j(E)$, then $\deg(x)=\deg(f) \cdot \deg(f(x))$, where $f$ denotes the natural map $X_0(n) \rightarrow X_0(\gcd(n,M))$.
\end{theorem}

\subsection{Linear disjointness of residue fields} 
\begin{proposition}[Menendez, \cite{Zonia}]\label{PropSurjectivePrimes}
Let $E$ be a non-CM elliptic curve over $k=\Q(j(E))$, let $\ell$ be a prime not contained in $S_E$, and let $a$ and $s$ be positive integers. Let $x \in X_0(a\ell^s)$ be a closed point with $j(x)=j(E)$. Then
\[
\deg(x)=\deg(f)\deg(f(x)),
\]
where $f\colon X_0(a \ell^s) \rightarrow X_0(a)$ is the natural map.
\end{proposition}

\begin{proof}
If $a=b\ell^t$ with $\ell \nmid b$, we let $g\colon X_0(a) \rightarrow X_0(b)$ and $h\colon  X_0(a \ell^s) \rightarrow X_0(b)$ so $h=g \circ f$. Since $\deg(h)=\deg(f)\cdot\deg(g)$, it suffices to consider the case where $\ell \nmid a$. Let $P \in E$ be a point of order $a \ell^s$ such that $x=[E, \langle P \rangle]$. As in the proof of \cite[Proposition 5.7]{belov}, the assumptions imply that
\[
[k(P):k(\ell^sP)]=\#\{Q\in E:\ell^sQ=\ell^sP, Q \text{ order } a\ell^s\}=(\ell^2-1)\ell^{2s-2}.
\]
That is, this extension is as large as possible. We will show this implies 
\[
[k(\langle P\rangle):k(\langle \ell^sP \rangle)]=\ell^{s-1}(\ell+1).
\]

Note $[k(\langle P\rangle):k(\langle \ell^sP \rangle)] \leq \deg(X_0(a\ell^s) \rightarrow X_0(a))=\ell^{s-1}(\ell+1)$. So suppose for the sake of contradiction that $[k(\langle P\rangle):k(\langle \ell^sP \rangle)] <\ell^{s-1}(\ell+1)$. This implies
\begin{align*}
(\ell^2-1)\ell^{2s-2}&=[k(P):k(\ell^sP)]\\
&=[k(P):k(\ell^sP)(\langle P \rangle)]\cdot [k(\ell^sP)(\langle P \rangle) : k(\ell^sP)]\\
&< \varphi(\ell^s)\cdot \ell^{s-1}(\ell+1)\\
&=(\ell^2-1)\ell^{2s-2},
\end{align*}
and we have a contradiction. So $[k(\langle P\rangle):k(\langle \ell^sP \rangle)]=\ell^{s-1}(\ell+1)$, as claimed. Since $k(\langle P \rangle)=\Q(x)$ and $k(\langle \ell^sP \rangle)=\Q(f(x))$, the result follows.
\end{proof}

\subsection{Proof of Theorem \ref{MainThm}}
Let $x=[E,\langle P\rangle ] \in X_0(n)$ be a closed point with $j(x)=j(E)$. By applying Proposition \ref{PropSurjectivePrimes} for each $\ell^s \parallel n$ with $\ell \not\in S_E$, we may assume $\Supp(n) \subset S_E$.
Let $d=\gcd(n,M)$. By assumption, the mod $n$ Galois representation is as large as possible given the image of the mod $d$ Galois representation. That is, for $n=p_1^{a_1} \cdots p_k^{a_k}$ and $d=p_1^{b_1} \cdots p_k^{b_k}$, we have
\begin{align*}
[k(E[n]):k(E[d])]&=\#\ker(\GL_2(\Z/n\Z) \rightarrow \GL_2(\Z/d\Z))\\
&=\frac{\# \GL_2(\Z/n\Z)}{\# \GL_2(\Z/d\Z)}\\
&=\displaystyle\prod_{p_i \mid d}p_i^{4(a_i-b_i)}\cdot \displaystyle\prod_{p_i \mid n, \,p_i\nmid{d}}p_i^{4(a_i-1)}(p_i^2-1)(p_i^2-p_i)
\end{align*}

Let $B(a)$ denote the subgroup of upper triangular matrices in $\GL_2(\Z/a\Z)$. Then
\begin{align*}
[k(E[n]):k(E[d])(\langle P \rangle)]&\leq\#\{M \in \GL_2(\Z/n\Z)\, |\, M \text{ upper triangular }, M \equiv I \pmod{d}\}\\
&=\# \ker(B(n) \rightarrow B(d))\\
&=\frac{\#B(n)}{\#B(d)}\\
&= \frac{\varphi(n)\cdot n \cdot \varphi(n)}{\varphi(d)\cdot d \cdot \varphi(d)}\\
&=\displaystyle \prod_{p_i \mid d} p_i^{3(a_i-b_i)} \cdot \displaystyle \prod_{p_i \mid n, p_i \nmid d} p_i^{3a_i-2}(p_i-1)^2.
\end{align*}
Comparing sizes shows that 
\begin{align*}
[k(E[d])(\langle P \rangle):k(E[d])]&\geq \displaystyle\prod_{p_i \mid n}p_i^{(a_i-b_i)}\cdot \displaystyle\prod_{p_i \mid n, \,p_i\nmid{d}}\left(1+\frac{1}{p_i}\right)\\
&=\deg(X_0(n) \rightarrow X_0(d)),
\end{align*}
and so $[k(\langle P \rangle):k(\langle \frac{n}{d}P \rangle)] \geq \deg(X_0(n) \rightarrow X_0(d))$. Since this latter extension has degree at most $\deg(X_0(n) \rightarrow X_0(d))$, equality holds. As $k(\langle P \rangle)=\Q(x)$ and $k(\langle  \frac{n}{d}P \rangle)=\Q(f(x))$, the result follows.

\end{appendices}

\bibliographystyle{amsplain}
\bibliography{references}

\end{document}